\documentclass[12pt,reqno]{amsart}
\usepackage[utf8]{inputenc}
\usepackage[T1]{fontenc}
\usepackage[usenames, dvipsnames]{color}
\usepackage{ulem}

\usepackage{dsfont, amsfonts, amsmath, amssymb,amscd, stmaryrd, latexsym, amsthm, dsfont}
\usepackage[frenchb,english]{babel}
\usepackage{enumerate}
\usepackage{longtable}
\usepackage{geometry}
\usepackage{float}
\usepackage{tikz}
\usetikzlibrary{shapes,arrows}
\usepackage{float}
\usepackage{tikz}
\usepackage{xypic}
\usetikzlibrary{shapes,arrows}

\usepackage{pifont}
\usepackage{float}
\usepackage{tikz}
\usepackage{xypic}
\usetikzlibrary{shapes,arrows}

\newtheorem{theorem}{Theorem}[section]
\newtheorem{lemma}[theorem]{Lemma}
\newtheorem{proposition}[theorem]{Proposition}
\newtheorem{corollary}[theorem]{Corollary}

\theoremstyle{remark}
\newtheorem{remark}[theorem]{\bf Remark}

\newtheorem{definition}[theorem]{\bf Definition}

\usepackage[pagebackref]{hyperref}
\renewcommand*{\backref}[1]{}\renewcommand*{\backrefalt}[4]{\ifcase #1 (\tt not cited)\or (\tt cited on page~#2)\else (\tt cited on pages~#2)\fi}

\usepackage{tikz-cd}
\usepackage{hyperref}
\hypersetup{
	colorlinks=true,
	urlcolor=blue,
	citecolor=blue,linkcolor=blue}
\def\NN{\mathbb{N}}
\def\RR{\mathds{R}}
\def\HH{I\!\! H}
\def\QQ{\mathbb{Q}}
\def\CC{\mathds{C}}
\def\ZZ{\mathbb{Z}}
\def\DD{\mathds{D}}
\def\OO{\mathcal{O}}
\def\kk{\mathds{k}}
\def\KK{\mathbb{K}}
\def\ho{\mathcal{H}_0^{\frac{h(d)}{2}}}
\def\LL{\mathbb{L}}
\def\L{\mathds{k}_2^{(2)}}
\def\M{\mathds{k}_2^{(1)}}
\def\k{\mathds{k}^{(*)}}
\def\l{\mathds{L}}

\usepackage{multirow}

\def\kk{\mathds{k}}
\begin{document}
	
	\def\NN{\mathbb{N}}
	\def\RR{\mathds{R}}
	\def\HH{I\!\! H}
	\def\QQ{\mathbb{Q}}
	\def\CC{\mathds{C}}
	\def\ZZ{\mathbb{Z}}
	\def\DD{\mathds{D}}
	\def\OO{\mathcal{O}}
	\def\kk{\mathds{k}}
	\def\KK{\mathbb{K}}
	\def\ho{\mathcal{H}_0^{\frac{h(d)}{2}}}
	\def\LL{\mathbb{L}}
	\def\L{\mathds{k}_2^{(2)}}
	\def\M{\mathds{k}_2^{(1)}}
	\def\k{\mathds{k}^{(*)}}
	\def\l{\mathds{L}}
	\def\2r{\mathrm{rank}}
	\def\rg{\mathrm{rank}}
	\def\C{\mathrm{C}}
	\def\Y{\mathbf{Y}}
	\def\V{\mathbf{V}}
	\def\vep{\varepsilon}

	\selectlanguage{english}
	
	
	\title[Greenberg's conjecture and Iwasawa module]{Greenberg's conjecture and Iwasawa module of Real biquadratic fields II}

	\author[M. M. Chems-Eddin]{M. M. Chems-Eddin}
	\address{Mohamed Mahmoud CHEMS-EDDIN: Department of Mathematics, Faculty of Sciences Dhar El Mahraz, Sidi Mohamed Ben Abdellah University, Fez,  Morocco}
	\email{2m.chemseddin@gmail.com}
	
	\author[H. El Mamry]{H. El Mamry}
	\address{Hamza EL MAMRY: Departement of Mathematics, Faculty of Sciences Dhar El Mahraz,  Sidi Mohamed Ben Abdellah University, Fez, Morocco}
	\email{Hamza.elmamry@usmba.ac.ma}


	\begin{abstract}
		In this paper we are interested in the stability of the $2$-rank of the class group in the cyclotomic $\ZZ_2$-extension of real biquadratic fields.
		In fact, we give  several   families of real biquadratic  fields $K$ such that $\rg(A(K)) =\rg(A_\infty(K))$ and  $\rg(A(K))\leq 3$, where $A(K)$ and $A_\infty(K)$ 
		are the $2$-class group and the $2$-Iwasawa module of $K$ respectively. Moreover, Greenberg's conjecture is verified for some new families of number fields; in particular, we determine the complete list   of all real biquadratic fields with trivial $2$-Iwasawa module. This work is a continuation of M. M. Chems-Eddin,	Greenberg's conjecture and Iwasawa module of real biquadratic fields I, J. Number Theory, 281  (2026),  224-266.
	\end{abstract}
	
	\selectlanguage{english}
	
	\maketitle


	\section{\bf Introduction}  Let  $k$ be a   number field and $\ell$ a prime number. Let $A_\ell(k)$  or simply $A(k)$ when $\ell=2$  (resp.   $E_k$) denote the $\ell$-class group (resp.   the unit group) of $k$.
	A   $\ZZ_\ell$-extension of $k$ is an infinite extension of $k$ denoted by  $k_\infty$ such that $\mathrm{Gal}(k_\infty/k)\simeq \ZZ_\ell$, where $\ZZ_\ell$ is the ring of $\ell$-adic numbers. For each $n\geq 1$, the extension $k_\infty/k$ contains a unique field denoted by $k_n$ and called the $n$th layer of the cyclotomic $\ZZ_\ell$-extension of $k$   of degree $\ell^n$. Furthermore, we have:
	$$k=k_0\subset k_1 \subset k_2 \subset\cdots \subset k_n \subset \cdots  \subset k_\infty=\bigcup_{n\geq 0} k_n.$$
	In particular, for an odd prime number $\ell$, let $ \QQ_{\ell, n}$   be the unique real subfield  of the cyclotomic field  $\QQ(\zeta_{\ell^{n+1}}) $ of degree $\ell^n$ over $\QQ$,  
	and for $\ell=2$,   let $  \QQ_{2, n}$ be the field $\QQ(2\cos( {2\pi}/{2^{n+2}}))$, for all $n\geq 1$.
	Then $k_\infty= \bigcup k_n$, where $k_n= k\QQ_{\ell, n}$,  is called the   cyclotomic $\ZZ_\ell$-extension of $k$.  
	The inverse limit $A(k_\infty):=\varprojlim A_\ell(k_n)$
	with respect to the norm maps is called the Iwasawa module for $k_\infty/k$. A spectacular result due to Iwasawa, affirms that  there exist integers $\lambda_k$,  $\mu_k\geq 0$ and  $\nu_k$, all independent of $n$, and  an integer $n_0$ such that:
	\begin{eqnarray}\label{iwasawa}h_\ell(k_n)=\ell^{\lambda_k n+\mu_k \ell^n+\nu_k},\end{eqnarray}
	for all $n\geq n_0$.  Here $h_\ell(k)$ denote the $\ell$-class number of a number field $k$. The integers $\lambda_k$,  $\mu_k$ and  $\nu_k$ are called the Iwasawa invariants of $k_\infty/k$
	(cf. \cite{iwasawa59} and \cite{washington1997introduction} for more details).
	
	In 1976,
	Greenberg conjectured that the invariants $\mu$ and $\lambda$ must be equal to $0$ for cyclotomic $\mathbb Z_\ell$-extension of totally real number
	fields (cf. \cite{Greenberg}).
	
	It was further proved by Ferrero and Washington     that the $\mu$-invariant always vanishes
	for the cyclotomic $\mathbb Z_\ell$-extension when the number field is abelian over the field $\mathbb{Q}$ of rational numbers (cf. \cite{FerreroWashington}). In other words, this  conjecture 
	is equivalent to $h_\ell(k_n)$ being uniformly bounded.
	Greenberg's conjecture is still open, with partial progress made by considering particular values of $\ell$ and specific families of number fields (especially the real quadratic fields), for example, we refer the reader to   \cite{7ChattopadhyayLaxmiSaikia,mizusawa3,mizusawa5,mizusawa2,mizusawa4,mouhib-mova}.
	Moreover, very recently some authors have taken interest  in the investigation of Greenberg's conjecture  for some particular families of real biquadratic fields. In fact, 
	Chems-Eddin, El Mahi and Ziane investigated this conjecture for  biquadratic fields of the form   $ \QQ(\sqrt{pq_1}, \sqrt{q_1q_2})$
	where $p\equiv 5\pmod 8$ and $q_1\equiv q_2\equiv3\pmod 4$  are three prime numbers such that $\ \left( \frac{p}{q_1} \right)=\ \left( \frac{p}{q_2} \right)$ (cf. \cite{chems24,Elmahi1,Elmahi2,Elmahi3}). The cyclotomic $\ZZ_2$-extension of   these fields were also the subject of a recent study by Azizi, Jerrari     and  Sbai   (cf. \cite{jerrarisbaiazizi}).
	Furthermore, in   \cite{Laxmi}
	Laxmi and      Saikia investigated this conjecture for biquadratic fields of the form   $ \QQ(\sqrt{p}, \sqrt{r})$ with $p \equiv 9 \mod{16}, \ r \equiv 3 \mod 4$ are two prime numbers such that $\ \left( \frac{p}{r} \right) = -1$, and $\left(\frac{2}{p}\right)_{4} = -1$.
	Thereafter, an extensive general study of real biquadratic fields is made in \cite{ChemseddinGreenbergConjectureI}, and this paper is a continuation of this work.
	Consider the following definition.
	\begin{definition}
		A number field $k$  is said QO-field if it is a quadratic extension of certain number field $k'$ whose class number is odd. We shall call $k'$ a base field of the QO-field $k$ and that $k/k'$ is a QO-extension. An extension of  QO-fields is an extension of number fields $L/M$ such that $L$ and $M$ are QO-fields.	\end{definition}
	
	In   \cite{ChemseddinGreenbergConjectureI}, the first author named considered  the following problems.
	\begin{center}
		{\bf  Problems: }
	\end{center}
	Let $K$ be a real  biquadratic number field such that $K_1/K$ is a ramified  extension of QO-fields, where $K_1=K(\sqrt{2})$ is the first layer of the cyclotomic $\ZZ_2$-extension of $K$. Consider the following problems:

	\begin{center}
		\noindent  {  Problem 1}: 	What are the real biquadratic fields $K$  such that  
		$$\rg(A(K_\infty))\leq 2   \text{ and   }    \rg( A(K_\infty))=\rg(A(K))?$$
	\end{center}
	\begin{center}
		\noindent  { Problem 2}: 	What is the structure of $A(K_\infty)$?$\qquad\qquad\qquad\quad$
		
	\end{center}

	\medskip
	
	According to \cite[p. 227]{ChemseddinGreenbergConjectureI}, the real biquadratic fields $K$ satisfying Problem 1 is in one of the following six forms:

	\begin{enumerate}[$    A)$]
		\item   $K=\QQ(\sqrt{q},\sqrt{d})$ where $d > 1$ is an odd  positive  square-free integer that is not divisible by $q$.

		\item  $K=\QQ(\sqrt{2q},\sqrt{d})$ where    $q\equiv3 \pmod 4$   and $d\equiv 1 \pmod 4$ is a  positive square-free integer.

		\item  $K=\QQ(\sqrt{q_1q_2},\sqrt{d})$ where  $q_1\equiv3 \pmod 4$,  $q_2\equiv3 \pmod 8$ are two prime numbers and $d\equiv 1\pmod 4$ is a positive    square-free integer that is not divisible by   $q_1q_2$.

		\item  $K=\QQ(\sqrt{q_1q_2},\sqrt{d})$ or $\QQ(\sqrt{q_1q_2},\sqrt{2d})$,  where  $q_1\equiv7 \pmod 8$,  $q_2\equiv3 \pmod 8$ are two prime numbers and $d\equiv 3\pmod 4$ is a positive        square-free integer that is not divisible by   $q_1q_2$.

		\item  $K=\QQ(\sqrt{q_1q_2},\sqrt{d})$ or $\QQ(\sqrt{q_1q_2},\sqrt{2d})$, where $q_1\equiv3 \pmod 8$,  $q_2\equiv3 \pmod 8$ are two prime numbers  and    $d\equiv 3 \pmod 4$ is a  positive square-free integer  that is not divisible by   $q_1q_2$.

		\item  $K=\QQ(\sqrt{p},\sqrt{d})$ where $p $ and $d$  satisfy one of the following conditions:
		\begin{enumerate}[$    a)$]
			
			\item  $p\equiv1 \pmod 4$   and $d=q_1q_2$ for two prime numbers $q_1\equiv q_2\equiv 3\pmod 4$  with   $\left(\frac{q_2}{p}\right)=-1$,
			\item  $p\equiv1 \pmod 4$   and $d=q_1q_2$ for two prime numbers $q_1\equiv 3\pmod 4$ and   $q_2\equiv 3 \pmod 8$,
			\item $p\equiv1 \pmod 4$   and $d\equiv 3 \pmod 4$ are two prime numbers such that $\left(\frac{p}{d}\right)=-1$ or $p\equiv5 \pmod 8$,
			\item $p\equiv1 \pmod 4$   and $d\equiv 1 \pmod 4$ are two prime numbers such that$\left(\frac{p}{d}\right)=-1$ or $[\left(\frac{p}{d}\right)=1$ and $\left(\frac{p}{d}\right)_4\not=\left(\frac{d}{p}\right)_4]$.
		\end{enumerate}

	\end{enumerate}
	Here, $p$ denotes a prime  number  congruent to $1\pmod 4$ and $q$, $q_1$, $q_2$ denote three primes numbers congruent to $3\pmod 4$. Let $r$ and $s$ be two prime numbers. Let us name $L$ any   real biquadratic field   of the form  
	\begin{center}
		{$ \displaystyle L:=\QQ(\sqrt{q_1q_2},\sqrt{r})\text{ or }\QQ(\sqrt{q_1q_2},\sqrt{rs}) \text{ where }   r\equiv 1\pmod8   \text{ and }  \left(\frac{q_1q_2}{r}\right)=1$}
	\end{center}

	\medskip
	
We note that 	\cite[Theorem 1.4]{ChemseddinGreenbergConjectureI} characterizes all real biquadratic fields of the forms $A)$, $B)$ and $C)$ that satisfy Problem 1.
	As a continuation of  the previous investigations, we prove the following first main result.
	
	\medskip
	
	\begin{theorem}[{\bf The First Main Theorem}]\label{maintheorem}	Let $K$ be a real biquadratic number field that is of the form $D)$ or $F)$ with $K\not=L$.  
		Then  $\rg(A(K_\infty))\leq 2$  and  $ \rg(A(K_\infty))=\rg( A(K))$ if and only if $K$ takes one of the following forms:
		\begin{enumerate}[$1)$]
			\item $K=\QQ(\sqrt{q_1q_2},\sqrt{\delta r})$, where  $q_1\equiv 7\pmod 8$ and $q_2\equiv 3\pmod 8$ and $r$ are prime numbers, $\delta\in\{1,q_1,q_2\}$ is such that 
			$\delta r\equiv 3\pmod4$ and we have one of the following congruence conditions:
			\begin{enumerate}[$\C1:$]
				\item  $r\equiv 3\pmod 8$ and $\left(\frac{q_1q_2}{ r}\right)=-1$,
				\item $r\equiv 3\pmod 8$ and $\left(\frac{q_1q_2}{ r}\right)=\left(\frac{q_1}{ r}\right)=1$,
				
				\item $r\equiv 5\pmod 8$ and  $\left(\frac{q_1}{ r}\right)=-1$.
			\end{enumerate}
			$\textbf{In this case, we have: } \rg(A_\infty(K))  =1.$
			
			\item $K=\QQ(\sqrt{q_1q_2},\sqrt{q_1 r})$, where  $q_1\equiv 7\pmod 8$ and $q_2\equiv 3\pmod 8$ and $r\equiv 1\pmod 4$ are prime numbers such that
			$\left(\frac{q_1q_2}{ r}\right)=\left(\frac{q_1}{ r}\right)=1$.
			
			\noindent\textbf{In this case, we have:} $\rg(A_\infty(K))  =2.$

			\item $K=\QQ(\sqrt{q_1q_2},\sqrt{\delta rs})$, where  $q_1\equiv 7\pmod 8$ and $q_2\equiv 3\pmod 8$, $r$ and $s$ are prime numbers, $\delta\in\{1,q_1,q_2\}$ is such that 
			$\delta rs\equiv 3\pmod4$,	$ \left(\frac{q_1q_2}{s}\right)=\left(\frac{q_1q_2}{r}\right)=-1$ and we have one of the following  conditions:
			\begin{enumerate}[$\C1:$]
				\item  $r\equiv 5\pmod 8$, $s\equiv 3\pmod 8$ and  $\left(\frac{q_1}{r}\right)=-1$,
				
				\item $r\equiv 3\pmod 8$, $s\equiv 3\pmod 8$ with   $\left(\frac{q_1}{r}\right)\not=\left(\frac{q_1}{s}\right)$.

			\end{enumerate}
			\textbf{In this case, we have: } $\rg(A_\infty(K))  =2.$
			\item $K=\QQ(\sqrt{q_1q_2},\sqrt{\delta rs})$, where  $q_1\equiv 7\pmod 8$ and $q_2\equiv 3\pmod 8$, $r$ and $s$ are prime numbers, $\delta\in\{1,q_1,q_2\}$ is such that 
			$\delta rs\equiv 3\pmod4$,	$ \left(\frac{q_1q_2}{r}\right)=-\left(\frac{q_1q_2}{s}\right)=-1$ and we have one of the following  conditions:

			\begin{enumerate}[$\C1:$]
				\item  $r\equiv s\equiv 3\pmod 8$  and  $\left(\frac{q_1}{r}\right)\neq\left(\frac{q_1}{s}\right)$,
				\item  $r\equiv 3\pmod 8$, $s\equiv5\pmod8$ and   $\left(\frac{q_1}{s}\right)=-1$,
				\item  $r\equiv 5\pmod 8$, $s\equiv3\pmod8$ and   $\left(\frac{q_1}{r}\right)=-1$.

			\end{enumerate}
			\textbf{In this case, we have: } $\rg(A_\infty(K))  =2.$

			\item $K=\QQ(\sqrt{q_1q_2},\sqrt{2\delta r})$, where  $q_1\equiv 7\pmod 8$ and $q_2\equiv 3\pmod 8$, $r$ is a prime number, $\delta\in\{1,q_1,q_2\}$ is such that 
			$\delta r\equiv 3\pmod4$	 and we have one of the following   conditions:

			\begin{enumerate}[$\C1:$]
				\item  $r\equiv 3\pmod 8$,
				\item  $r\equiv 5\pmod 8$  and    $\left(\frac{q_1}{r}\right)=-1$.
			\end{enumerate}
			\textbf{In this case, we have: } $\rg(A_\infty(K))  =1.$
			
			\item $K=\QQ(\sqrt{q_1q_2},\sqrt{2q_1 r})$, where  $q_1\equiv 7\pmod 8$ and $q_2\equiv 3\pmod 8$, $r$ is a prime number such that 
			$r\equiv 5\pmod 8$  and  $\left(\frac{q_1q_2}{r}\right)=\left(\frac{q_1}{r}\right)=1$.

			\noindent	\textbf{In this case, we have: } $\rg(A_\infty(K))  =2.$

			\item $K=\QQ(\sqrt{q_1q_2},\sqrt{2\delta rs})$, where  $q_1\equiv 7\pmod 8$ and $q_2\equiv 3\pmod 8$, $r$ and $s$ are  prime numbers, $\delta\in\{1,q_1,q_2\}$ is such that 
			$\delta rs\equiv 3\pmod8$,      $\left(\frac{q_1q_2}{r}\right)=\left(\frac{q_1q_2}{s}\right)=-1$ and we have one of the following   conditions:\begin{enumerate}[$\C1:$]
				\item $r\equiv 5\pmod8$, $s\equiv 3\pmod8$ and $\left(\frac{q_1}{r}\right)=-1$,
				\item $r\equiv s \equiv 3\pmod8$  and $\left(\frac{q_1}{r}\right)\neq \left(\frac{q_1}{s}\right) $.
			\end{enumerate}

			\noindent\textbf{In this case, we have: } $\rg(A_\infty(K))  =2.$

			\item $K=\QQ(\sqrt{q_1q_2},\sqrt{2\delta rs})$, where  $q_1\equiv 7\pmod 8$ and $q_2\equiv 3\pmod 8$, $r$ and $s$ are  prime numbers, $\delta\in\{1,q_1,q_2\}$ is such that 
			$\delta rs\equiv 3\pmod8$, and    $\left(\frac{q_1q_2}{r}\right)=-\left(\frac{q_1q_2}{s}\right)=-1$ and we have one of the following   conditions:\begin{enumerate}[$\C1:$]
				\item $r\equiv s \equiv 3\pmod8$  and $\left(\frac{q_1}{r}\right)\neq \left(\frac{q_1}{s}\right) $,
				\item $r\equiv 5\pmod8$, $s\equiv 3\pmod8$ and $\left(\frac{q_1}{r}\right)=-1$,
				\item $r\equiv 3\pmod8$, $s\equiv 5\pmod8$ and $\left(\frac{q_1}{s}\right)=-1$.
			\end{enumerate}

			\noindent	\textbf{In this case, we have: } $\rg(A_\infty(K))  =2.$

			\item $K=\QQ(\sqrt{q_1q_2},\sqrt{2\delta rs})$, where  $q_1\equiv 7\pmod 8$ and $q_2\equiv 3\pmod 8$, $r$ and $s$ are  prime numbers, $\delta\in\{1,q_1,q_2\}$ is such that 
			$\delta rs\equiv 3\pmod8$, and    $\left(\frac{q_1q_2}{r}\right)=\left(\frac{q_1q_2}{s}\right)=1$ and one of the following   conditions :\begin{enumerate}[$\bullet$]
				\item $r\equiv s \equiv 3\pmod8$ and $\left(\frac{q_1}{r}\right)=-\left(\frac{q_1}{s}\right)=1$,
				\item $r\equiv 5\pmod8$ , $s\equiv 3\pmod8$ and $\left(\frac{q_1}{r}\right)=-1$.
			\end{enumerate}

		\noindent	\textbf{In this case, we have: } $\rg(A_\infty(K))  =2.$

			\item $K=\QQ(\sqrt{p},\sqrt{q_1q_2})$, where  $p\equiv1\pmod4$, $q_1\equiv 3 \pmod4$, $q_2\equiv 3 \pmod8$ are distinct prime numbers such that   $\left(\frac{q_2}{p}\right)=-1$  and we have one of the following   conditions:
			\begin{enumerate}[$\C1:$]
				\item $p\equiv5\pmod8, q_1\equiv3\pmod8 $ and $\left(\frac{q_1}{p}\right)=1$,
				\item $p\equiv5\pmod8, q_1\equiv7\pmod8 $ and $\left(\frac{q_1}{p}\right)=-1$.
			\end{enumerate}
			$\textbf{In this case, we have: } \rg(A_\infty(K))  =0.$

			\item $K=\QQ(\sqrt{p},\sqrt{d})$, where  $p\equiv 5\pmod8$ and $d\equiv  3 \pmod4$ are two prime numbers.
			
			\noindent	\textbf{In this case, we have: } $\rg(A_\infty(K))  =0.$

			
			
			%
			\item 	$K=\QQ(\sqrt{p},\sqrt{d})$, where $p\equiv d \equiv 1 \pmod4$ are two prime numbers satisfying the following conditions:
			\begin{enumerate}[$a)$]
				\item $\left(\frac{p}{d}\right)=-1$ or $[\left(\frac{p}{d}\right)=1$ and $\left(\frac{p}{d}\right)_4\not=\left(\frac{d}{p}\right)_4]$,
				\item     
				$h_2(2pd)=4$ and
				\item  $q(k_1)=1$, where $k_1=\QQ(\sqrt{pd}, \sqrt{2})$.
			\end{enumerate}
			\textbf{In this case, we have: } $\rg(A_\infty(K))  =0.$
		\end{enumerate}

	\end{theorem}


	\medskip
	
	The plan of this paper is the following. In Section \ref{Sec.1} we prove Theorem \ref{maintheorem}. In Section \ref{Sec.2} we give some families or real biquadratic fields such  
	that $ \rg(A(K_\infty))=\rg( A(K))=3$.  
	The last section is dedicated to the investigation of the structure of the $2$-Iwasawa of real biquadratic fields. In fact, we give the list of all real biquadratic field with trivial $2$-Iwasawa module. Moreover, we give some new families of real biquadratic fields that satisfy Greenberg's Conjecture.

	\section{\bf The proof of   Theorem \ref{maintheorem}}\label{Sec.1}
	
Let us start by collecting some preliminary results that will be used later.	
	\subsection{Prerequisites }  Let us start by recalling the following useful results.

	\begin{lemma}[\cite{Ku-50}]\label{wada's f.}
		Let $k$ be a multiquadratic number field of degree $2^n$, $n\geq 2$,  and $k_i$ be the $s=2^n-1$ quadratic subfields of $k$. Then
		$$h(k)=\frac{1}{2^v}q(k)\prod_{i=1}^{s}h(k_i),$$
		where  $ q(k)=[E_k: \prod_{i=1}^{s}E_{k_i}]$ and   $$     v=\left\{ \begin{array}{cl}
			n(2^{n-1}-1); &\text{ if } k \text{ is real, }\\
			(n-1)(2^{n-2}-1)+2^{n-1}-1 & \text{ if } k \text{ is imaginary.}
		\end{array}\right.$$
	\end{lemma}
	\bigskip
	To compute that unit index  $q(k)$ appearing in the above lemma, it is useful  to
	recall the following  method given in    \cite{wada}, that describes a fundamental system  of units of a real  multiquadratic field $k_0$. Let  $\sigma_1$ and 
	$\sigma_2$ be two distinct elements of order $2$ of the Galois group of $k_0/\mathbb{Q}$. Let $k_1$, $k_2$ and $k_3$ be the three subextensions of $k_0$ invariant by  $\sigma_1$,
	$\sigma_2$ and $\sigma_3= \sigma_1\sigma_2$, respectively. Let $\varepsilon$ denote a unit of $k_0$. Then \label{algo wada}
	$$\varepsilon^2=\varepsilon\varepsilon^{\sigma_1}  \varepsilon\varepsilon^{\sigma_2}(\varepsilon^{\sigma_1}\varepsilon^{\sigma_2})^{-1},$$
	and we have, $\varepsilon\varepsilon^{\sigma_1}\in E_{k_1}$, $\varepsilon\varepsilon^{\sigma_2}\in E_{k_2}$  and $\varepsilon^{\sigma_1}\varepsilon^{\sigma_2}\in E_{k_3}$.
	It follows that the unit group of $k_0$  
	is generated by the elements of  $E_{k_1}$, $E_{k_2}$ and $E_{k_3}$, and the square roots of elements of   $E_{k_1}E_{k_2}E_{k_3}$ which are perfect squares in $k_0$.
	\bigskip

	The following is results a particular case of the well known Fukuda's Theorem  \cite{fukuda}.

	\begin{lemma}[\cite{fukuda}]\label{lm fukuda}
		Let $k_\infty/k$ be a $\mathbb{Z}_2$-extension and $n_0$  an integer such that any prime of $k_\infty$ which is ramified in $k_\infty/k$ is totally ramified in $k_\infty/k_{n_0}$.
		\begin{enumerate}[\rm $1)$]
			\item If there exists an integer $n\geq n_0$ such that   $h_2(k_n)=h_2(k_{n+1})$, then $h_2(k_n)=h_2(k_{m})$ for all $m\geq n$.
			\item If there exists an integer $n\geq n_0$ such that $\rg( A(k_n))= \rg(A(k_{n+1}))$, then
			$\rg(A(k_{m}))= \rg(A(k_{n}))$ for all $m\geq n$.
		\end{enumerate}
	\end{lemma}

	\begin{lemma}[\cite{Qinred}, Lemma 2.4]\label{AmbiguousClassNumberFormula} Let $k/k'$  be a QO-extension of number fields. Then the rank of the $2$-class group of $k$ is given by
		$$\rg({A(k)})=t_{k/k'}-1-e_{k/k'},$$
		where  $t_{k/k'}$ is the number of  ramified primes (finite or infinite) in the extension  $k/k'$ and $e_{k/k'}$ is  defined by   $2^{e_{k/k'}}=[E_{k'}:E_{k'} \cap N_{k/k'}(k^*)]$.
	\end{lemma}

	\bigskip
	
	Let  $k^{(1)}$  be the Hilbert $2$-class field of a number field $k$, i.e.  the maximal unramified  abelian   extension
	of $k$ whose degree over $k$ is a power of $2$. Put $k^{(0)} = k$ and let $k^{(i)}$ denote the Hilbert $2$-class field of $k^{(i-1)}$
	for any integer $i\geq 1$. Then the sequence of fields
	$$k=k^{(0)} \subset k^{(1)} \subset  k^{(2)}  \subset \cdots\subset k^{(i)} \cdots \subset  \bigcup_{i\geq 0} k^{(i)}=\mathcal{L}(k)$$
	is called   the $2$-class field tower of $k$. If for all $i\geq1$,  $k^{(i)}\neq k^{(i-1)}$, the tower is said to be infinite, otherwise the tower is said to be  finite, and the minimal integer $i$ satisfying the condition $k^{(i)}= k^{(i-1)}$ is called the length of the tower.  The following nice characterization is due to
	Benjamin,	  Lemmermeyer and	 Snyder.

	\begin{proposition}[\cite{BLS98}, Proposition 7]\label{LemBenjShn}
		Let $k$ be a number field such that $A(k)\simeq (2^m, 2^n)$ for some positive integers $m$ and $n$. If there is an unramified
		quadratic extension of $k$  with $2$-class number $2^{m+n-1}$,  then all three unramified quadratic extensions of $k$ have $2$-class number
		$2^{m+n-1}$,  and the $2$-class field tower of $k$ terminates at $k^{(1)}$.
		
		Conversely, if the $2$-class field tower of $k$ terminates at $k^{(1)}$, then all three
		unramified quadratic extensions of $k$ have $2$-class number $2^{m+n-1}$.
	\end{proposition}	
	
	Moreover,
	   let $k$ be a number field such that $A(k)\simeq\ZZ/2 \ZZ\times\ZZ/2^n \ZZ$, where $n\geq2$ is a natural number.
	Put $G_k=\mathrm{Gal}(\mathcal{L}(k)/k)$. So by class field theory, we have 
	$G_k/G_k'\simeq\ZZ/2 \ZZ\times\ZZ/2^n \ZZ$ and  $G_k=\langle a,
	b\rangle$ such that $a^2\equiv b^{2^n}\equiv 1\mod G_k'$ and $A(k)=\langle \mathfrak{c}, \mathfrak{d}\rangle\simeq \langle aG_k',
	bG'\rangle$ where $(\mathfrak{c}, k^{(2)}/k)=aG_k'$
	and $(\mathfrak{d}, k^{(2)}/k)=bG_k'$ with $(\ .\ ,
	k^{(2)}/k)$ is the Artin symbol in the extension
	$k^{(2)}/k$. Therefore, there exist three normal subgroups of $G_k$ of index $2$ denoted  $H_{1, 2}$, $H_{2, 2}$ and
	$H_{3, 2}$ such that
	\begin{center}
		$H_{1, 2}=\langle b, G_k'\rangle$, $H_{2, 2}=\langle ab, G_k'\rangle$ and
		$H_{3, 2}=\langle a, b^2, G_k'\rangle.$
	\end{center}
	Furthermore, there exist three normal subgroups of $G_k$ of index
	$4$  denoted  $H_{1, 4}$,
	$H_{2, 4}$ and $H_{3, 4}$ such that
	$$H_{1, 4}=\langle a,b^4, G_k'\rangle, \quad H_{2, 4}=\langle ab^2, G_k'\rangle \text{ and } H_{3, 4}=\langle b^2, G_k'\rangle.$$
	Each subgroup $H_{i,j}$ of $A(k)$
	corresponds to 	an unramified extension    ${K}_{i,j}$ of $k$ contained in
	$k^{(1)}$ such that $A(k)/H_{i,j}\simeq
	\operatorname{Gal}(K_{i,j}/k)$ and
	$H_{i,j}=\mathcal{N}_{{K}_{i,j}/k}(A(K_{i,j}) )$. The Hilbert $2$-class field tower of $k$ can be schematized as follows:
	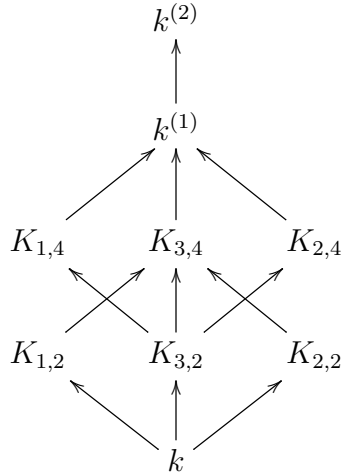
\begin{figure}[H]
		$$
		\xymatrix{
			& k^{(2)} \ar@{<-}[d] & \\
			& k^{(1)}\ar@{<-}[ld]\ar@{<-}[d]\ar@{<-}[rd]\\
			{K_{1,  4}}\ar@{<-}[rd]&\ar@{<-}[ld] {K_{3,  4}}\ar@{<-}[d]\ar@{<-}[rd]  & {K_{2,  4}}\ar@{<-}[ld]\\
			{K_{1,  2}}\ar@{<-}[rd]& {K_{3,  2}} \ar@{<-}[d]  & {K_{2,  2}}\ar@{<-}[ld]\\
			&k
		}
		$$
		\caption{\label{Fig2} Unramified extensions}
	\end{figure}
	
	   The following interesting result	is  due to
	Aaboun and   Zekhnini.

	\begin{theorem}[\cite{aaboune}, Theorem 4.10]\label{AabounePrzekhini}
		The following assertions are equivalent:
		\begin{enumerate}[\rm $1)$]
			\item $G_k$ is abelian,
			\item The  $2$-class group of ${K_{3,  2}}$ is isomorphic to  $\ZZ/2 \ZZ\times\ZZ/2^{n-1} \ZZ$,
			\item The  $2$-class number of ${K_{3,  2}}$ equals $2^n$,
			\item The $2$-class group of  ${K_{i, 2}}$ is cyclic of order $2^n$ with $i=1$ or $2$,
			\item The  $2$-class number of ${K_{i,  2}}$ equals $2^n$ with $i=1$ or $2$,
			\item The Hilbert $2$-class field tower of $k$ stops at $k^{(1)}$.
		\end{enumerate}
	\end{theorem}

	\subsection{The proof of   Theorem \ref{maintheorem}}
	Let us start with   some lemmas concerning the $2$-rank of the class group of real biquadratic fields.
	 	The following remark  is useful to simplify the statement of   Lemmas \ref{lemmaq1q2Bd=3mod4} and   \ref{lemma2qC}.
	\begin{remark}\label{remcondic1c2}
		Let $q_1\equiv3 \pmod 4$,   $q_2\equiv3 \pmod 8$, $r$ and $s$ be four prime numbers.   Let us defined  the following two conditions:  
		\begin{enumerate}[$\C1:$]
			\item    
			$\left(\frac{q_1q_2}{r}\right)=-\left(\frac{q_1q_2}{s}\right)=1$ and   we have one of the following conditions :\begin{enumerate}[$\bullet$]
				\item $q_1\equiv7 \pmod 8$, $r\equiv1 \pmod 4$ and $\left(\frac{q_1}{r}\right)=1$, or
				\item $q_1\equiv3 \pmod 8$ and $\left(\frac{q_1}{r}\right)=\left(\frac{-1}{r}\right)$.
			\end{enumerate}

			\item   $q_1\equiv7 \pmod 8$, $\left(\frac{q_1q_2}{r}\right)=\left(\frac{q_1q_2}{s}\right)=1$ and we have one of the following conditions :\begin{enumerate}[$\bullet$]
				\item $r\equiv s\equiv1 \pmod 4$, 
				
				\item $r\equiv1 \pmod 4$ and $\left(\frac{q_1}{r}\right)=1$,
				\item  $\left(\frac{q_1}{r}\right)=\left(\frac{q_1}{s}\right)=1$, or
				
				\item $ r \equiv s \equiv 3 \pmod 4$ and  $\left(\frac{q_1}{r}\right)= \left(\frac{q_1}{s}\right)=-1.$


			\end{enumerate}
			

			\item   
			$\left(\frac{q_1q_2}{r}\right)=-\left(\frac{q_1q_2}{s}\right)=1$ and we have one of the following conditions:\begin{enumerate}[$\bullet$]
				\item $q_1\equiv7 \pmod 8$, $r\equiv1 \pmod 4$ and $\left(\frac{q_1}{r}\right)=1$,
				\item $q_1\equiv3 \pmod 8$ and $\left(\frac{q_1}{r}\right)=1$.
			\end{enumerate}
		\end{enumerate}
	\end{remark}


	\begin{lemma}\label{lemmaq1q2Bd=3mod4}Let   $K=\QQ(\sqrt{q_1q_2},\sqrt{d})$, where $q_1\equiv3 \pmod 4$,  $q_2\equiv3 \pmod 8$ are two prime numbers  and  ${\color{blue} d \equiv 3\pmod 4}$ is a  positive  square-free integer that is not divisible by $q_1q_2$. Let   $\delta\in \{1,q_1,q_2\}$.  
		Then $\2r(A(K))\leq 2$ if and only if $K$ takes one of the following forms:
		\begin{enumerate}[$1)$]
			\item  $K=\QQ(\sqrt{q_1q_2},\sqrt{d})$, where   $d=\delta r\equiv 3\pmod 4 $ with $r$  is a prime  number. {\bf In this case,} we have $\2r(A(K))\in \{1,2\}$. More precisely, $\2r(A(K))=1$ if and only if {we have one of} the following conditions:
			\begin{enumerate}[$	 a) $]
				\item    $ \left(\frac{q_1q_2}{r}\right)=-1$, 
				\item     $ \left(\frac{q_1q_2}{r}\right)=1$ and  $\left(\frac{q_1}{ r}\right)\neq \left(\frac{-1}{ r}\right)$.
				
			\end{enumerate}

			\item    $K=\QQ(\sqrt{q_1q_2},\sqrt{d})$  with  $d=\delta rs\equiv 3\pmod 4 $, where $r$  and $s$ are prime  numbers which do not satisfy any of the conditions 
			$\C1$ and $\C2$ defined in   Remark \ref{remcondic1c2}. {\bf In this case}, we have  $ \2r(A(K))\in \{1,2\}$. More precisely,  $ \2r(A(K))=2 $ if and only we have one of the following conditions:
			\begin{enumerate}[$a)$]
				\item   $ \left(\frac{q_1q_2}{r}\right)=\left(\frac{q_1q_2}{s}\right)=-1$,  
				
				\item  $q_1\equiv 7\pmod 8$,   $ \left(\frac{q_1q_2}{r}\right)=-\left(\frac{q_1q_2}{s}\right)=-1$ and   $[\left(\frac{q_1}{s}\right)=-1$ or $\left(\frac{-1}{s}\right)=-1]$,
				
				\item    $q_1\equiv 7\pmod 8$,     $ \left(\frac{q_1q_2}{r}\right)=\left(\frac{q_1q_2}{s}\right)=1$ and  $[ \left(\frac{q_1 }{ s}\right)=-\left(\frac{q_1 }{ r}\right)=\left(\frac{-1 }{ r}\right)=-1$  or  $\left(\frac{q_1}{ s}\right)=\left(\frac{-1}{ r}\right)=-\left(\frac{-1}{ s}\right)=-1]$.

			\end{enumerate}		
		\end{enumerate}	
	\end{lemma}
	\begin{proof}
		Assume that $K=\QQ(\sqrt{q_1q_2},\sqrt{d})$ where   $q_1\equiv3 \pmod 4$,  $q_2\equiv3 \pmod 8$ are two prime numbers  and {\color{blue}$d\equiv 3\pmod 4$} is any       positive   square-free integer that is not divisible by   $q_1q_2$.
		Let $t_{K/F}$ be the number of prime ideals of $ F=\QQ(\sqrt{q_1q_2})$ that are ramified in $K$. 
		Thus,  $ \2r(A(K))=t_{K/F}-1-e_{K/F}$. 
		As  $e_{K/F}\in\{0,1,2\}$, then the inequality $\2r(A(K))\leq 2$ implies that $t_{K/F}\leq 5$. Therefore, $d$ takes one of the following five forms
		\begin{eqnarray*}
			&\delta r, &\ \delta rs \text{ whenever }  q_1q_2\equiv 5\pmod 8 \text{ or } \left(\frac{q_1q_2}{s}\right)=-1 , \\
			& \delta  rst&  \text{ whenever }  q_1q_2\equiv 1\pmod 8 \text{ and }  \left(\frac{q_1q_2}{r}\right)= \left(\frac{q_1q_2}{s}\right)=\left(\frac{q_1q_2}{t}\right)=-1  ,\\
			& \delta  rst&  \text{ whenever }  q_1q_2\equiv 5\pmod 8 \text{ and }     \left(\frac{q_1q_2}{s}\right)=\left(\frac{q_1q_2}{t}\right)=-1,\\
			&\delta rstk& \text{ whenever }  q_1q_2\equiv 5\pmod 8 \text{ and }  \left(\frac{q_1q_2}{r}\right)=\left(\frac{q_1q_2}{s}\right)=\left(\frac{q_1q_2}{t}\right)=\left(\frac{q_1q_2}{k}\right)=-1,
		\end{eqnarray*}
		where $\delta\in \{1,q_1,q_2\}$ and   $r$, $s$, $t$     and $k$ are prime numbers all different of $q_1$ and $q_2$.   
		\begin{enumerate}[$\blacktriangleright$]
			\item 	  Let  $d=\delta r$.   
			We have:
			\begin{enumerate}[$\star$]
				\item If $q_1q_2\equiv 5\pmod 8$ and  $ \left(\frac{q_1q_2}{r}\right)=-1$, then $ \2r(A(K))=2-1-e_{K/F}$. Thus,     \cite[Theorems 3.3]{Azmouh2-rank} gives
				$ \2r(A(K))=1$.

				\item If $q_1q_2\equiv 5\pmod 8$ and  $ \left(\frac{q_1q_2}{r}\right)=1$, then \cite[Theorems 3.3]{Azmouh2-rank} implies that    $ \2r(A(K))\in \{1,2  \}$, more precisely,  $ \2r(A(K))=2$ if and only if $\left(\frac{q_1  }{ r}\right)=\left(\frac{-1 }{ r}\right)=1$.

				\item If $q_1q_2\equiv 1\pmod 8$ and  $ \left(\frac{q_1q_2}{r}\right)=-1$, then  $ \2r(A(K))=3-1-e_{K/F}$ and \cite[Theorems 3.4]{Azmouh2-rank} implies that    $ \2r(A(K))=1$. 

				\item If $q_1q_2\equiv 1\pmod 8$ and  $ \left(\frac{q_1q_2}{r}\right)=1$, then $ \2r(A(K))=4-1-e_{K/F}$ and  \cite[Theorems 3.4]{Azmouh2-rank} gives
				$ \2r(A(K))\in \{1,2  \}$, more precisely,  $ \2r(A(K))=1$ if and only if $\left(\frac{q_1 }{r}\right)\left(\frac{-1 }{r}\right)=-1$.
				

			\end{enumerate}

			\item  Let  $d=  \delta rs$ with   $q_1q_2\equiv 5\pmod 8 \text{ or } \left(\frac{q_1q_2}{s}\right)=-1$. 
			\begin{enumerate}[$\star$]
				\item Assume that  $q_1q_2\equiv 5\pmod 8 \text{ and  } \left(\frac{q_1q_2}{s}\right)=-1$.  
				If $\left(\frac{q_1q_2}{r}\right)=-1$, then $ \2r(A(K))=3-1-e_{K/F}$ and   \cite[Theorems 3.3]{Azmouh2-rank} gives $ \2r(A(K)) =2$.

				If $\left(\frac{q_1q_2}{r}\right)=1$, then   \cite[Theorems 3.3]{Azmouh2-rank} gives again
				$ \2r(A(K))=4-1-e_{K/F}=3-e_{K/F}\in \{3,2\}$, more precisely, $ \2r(A(K))=2$ if and only if $\left(\frac{q_1}{r}\right)=-1$ or $\left(\frac{-1}{r}\right)=-1$.

				\item Assume that $q_1q_2\equiv 5\pmod 8 \text{ and  } \left(\frac{q_1q_2}{s}\right)=1$. 
				If $\left(\frac{q_1q_2}{r}\right)=-1$, then  we have
				$ \2r(A(K)) \in \{3,2\}$ and $ \2r(A(K))=2$ if and only if $\left(\frac{q_1}{s}\right)=-1$ or $\left(\frac{-1}{s}\right)=-1$.
				If $\left(\frac{q_1q_2}{r}\right)=1$, then \cite[Theorems 3.3]{Azmouh2-rank}, we have
				$ \2r(A(K))=5-1-e_{K/F}=4-e_{K/F}\in \{2,3,4\}$ and $ \2r(A(K))=2$ if and only if $e_{K/F}=2$, which is equivalent to $[ \left(\frac{-1 }{ r}\right)=-1$ and  $\left(\frac{q_1 }{ s}\right)=-\left(\frac{q_1 }{ r}\right)=-1 ]$ or  $[\left(\frac{-1}{ r}\right)=-\left(\frac{-1}{s}\right)=-1$ and  $ \left(\frac{q_1}{s}\right)=-1 ]$.

				\item Assume that $q_1q_2\equiv 1\pmod 8 \text{ and  } \left(\frac{q_1q_2}{s}\right)=-1$. 
				
				\noindent $*)$  If $\left(\frac{q_1q_2}{r}\right)=-1$, then    $ \2r(A(K))=4-1-e_{K/F}=3-e_{K/F}$ and \cite[Theorems 3.4]{Azmouh2-rank} gives $ \2r(A(K))=2$.
				
				\noindent $*)$ Now if  $\left(\frac{q_1q_2}{r}\right)= 1$, then   we have
				$ \2r(A(K))=5-1-e_{K/F}=4-e_{K/F}\in \{2,3\} $ and \cite[Theorems 3.4]{Azmouh2-rank} gives $ \2r(A(K))=2$ if and only if  $ \left(\frac{-1}{ r}\right)\not=\left(\frac{q_1 }{ r}\right)$.
			\end{enumerate}
			
			\item In the remaining cases, we check similarly that   we have		$ \2r(A(K))   \geq 3$.
		\end{enumerate}	
	\end{proof}

Let $	\left(\frac{-,\, -}{ \cdot}\right)$ denote the norm residue symbol.
	The following lemma corrects  an error in \cite[Theorem  3.4]{Azmouh2-rank} concerning the fields $\QQ(\sqrt{q_1q_2},\sqrt{2d})	$, where $d\equiv3\pmod 4$ and $q_1q_2\equiv1\pmod8$.
	\begin{lemma}\label{corrlemma}
		Let $d\equiv3\pmod 4$ be a positive square-free integer and $K=\QQ(\sqrt{q_1q_2},\sqrt{2d})	$. Assume that $q_1\equiv q_2\equiv 3\pmod8$. Then 
		$e_{K/F}\in\{1,2\}$ and $e_{K/F}=2$ if and only if  $\left(\frac{q_1}{r}\right)=-1$, for some prime divisor $r$ of $d$ such that  $\left(\frac{q_1q_2}{r}\right)=1$.
		Here $F=\QQ(\sqrt{q_1q_2})	$.
	\end{lemma}
	\begin{proof}Let $r$ be an odd prime divisor of $d$ and $\mathcal{R}$ be a prime ideal of $F$ lying above $r$. Assume that $\left(\frac{q_1q_2}{r}\right)=1$. 
		As in the proof of \cite[Theorem  3.3]{Azmouh2-rank}, we have:
		$$	\left(\frac{\vep_{q_1q_2},\, 2d}{ \mathcal R}\right)=\left(\frac{q_1}{r}\right), \quad 
		\left(\frac{-1,\, 2d}{ \mathcal R}\right)=\left(\frac{-1}{r}\right)$$
		
		$$	\left(\frac{\vep_{q_1q_2},\, 2d}{ \mathfrak 2}\right)=1, \quad 
		\left(\frac{-1,\, 2d}{ \mathfrak 2}\right)=-1.$$	
		Thus, $e_{K/F}\in\{1,2\}$ and $e_{K/F}=2$ if and only if    $\left(\frac{q_1}{r}\right)=-1$.
		We similarly check that if all prime divisors $r$ of $d$ are such that 	 $\left(\frac{q_1q_2}{r}\right)=-1$, then $e_{K/F}=1$.
	\end{proof}

	\begin{lemma}\label{lemma2qC} Let $K=\QQ(\sqrt{q_1q_2},\sqrt{2d})$, where $q_1\equiv3 \pmod 4$,  $q_2\equiv3 \pmod 8$ are two prime numbers  and    $d\equiv 3 \pmod 4$ is a  positive square-free integer  that is not divisible by   $q_1q_2$  
		and $\delta\in \{1,q_1,q_2\}$.
		Then $\2r(A(K))\leq 2$ if and only if $K$ takes one of the following forms:
		\begin{enumerate}[$1)$]
			\item$K=\QQ(\sqrt{q_1q_2},\sqrt{2\delta r})$ where r is an odd prime number. {\bf In  this case}, we have $\2r(A(K))\in \{1,2\}$. More precisely, $\2r(A(K))=2$ if and only if one of the following conditions holds:
			\begin{enumerate}[$a)$]
				\item $q_1\equiv3\pmod8$ and $\left(\frac{q_1q_2}{r}\right)=\left(\frac{q_1}{r}\right)=1,$
				\item $q_1\equiv7\pmod8$ , $r\equiv1\pmod4$  and $\left(\frac{q_1q_2}{r}\right)=\left(\frac{q_1}{r}\right)=1.$

			\end{enumerate}



			\item $K=\QQ(\sqrt{q_1q_2},\sqrt{2\delta rs})$ where $r$ and $s$ are two odd prime numbers do not satisfy any of the conditions $\C2$ and $\C3$ 
			defined in Remark \ref{remcondic1c2}. {\bf In this case}, we have $ \2r(A(K))\in \{1,2\}$. More precisely,  $ \2r(A(K))=2 $ if and only if one of the following conditions holds:
			\begin{enumerate}[$a)$]
				\item $\left(\frac{q_1q_2}{r}\right)=\left(\frac{q_1q_2}{s}\right)=-1$,

				\item $q_1\equiv 3 \pmod 8$ , $\left(\frac{q_1q_2}{r}\right)=-\left(\frac{q_1q_2}{s}\right)=1$ and $\left(\frac{q_1}{r}\right)=-1$,

				\item $q_1\equiv 7 \pmod 8$, $\left(\frac{q_1q_2}{r}\right)=-\left(\frac{q_1q_2}{s}\right)=1$ and $[r\equiv3\pmod4$ or $\left(\frac{q_1}{r}\right)=-1]$,
				

				\item   $q_1\equiv 7 \pmod 8$, $\left(\frac{q_1q_2}{r}\right)=\left(\frac{q_1q_2}{s}\right)=1$ and we have one of the following conditions:
				\begin{enumerate}[$i)$]
					\item $r\equiv 3\pmod4$ and $\left(\frac{q_1}{r}\right)=-\left(\frac{q_1}{s}\right)=1$,
					
					\item $r\equiv 3\pmod4$,  $s\equiv1\pmod4$ and $\left(\frac{q_1}{s}\right)=-1$.
				\end{enumerate}
			\end{enumerate}
			

		\end{enumerate}	
	\end{lemma}
	\begin{proof}
		Assume that $K=\QQ(\sqrt{q_1q_2},\sqrt{2d})$ where   $q_1\equiv3 \pmod 4$,  $q_2\equiv3 \pmod 8$ are two prime numbers  and {\color{red}$d\equiv 3\pmod 4$} is any       positive   square-free integer that is not divisible by   $q_1q_2$.
		Let $t_{K/F}$ be the number of prime ideals of $ F=\QQ(\sqrt{q_1q_2})$ that are ramified in $K$.   
		Thus,  $ \2r(A(K))=t_{K/F}-1-e_{K/F}$. 
		As  $e_{K/F}\in\{0,1,2\}$, then the inequality $\2r(A(K))\leq 2$ implies that $t_{K/F}\leq 5$. Therefore, $d$ takes one of the following five forms
		\begin{eqnarray*}
			&\delta r, &\ \delta rs \text{ whenever }  q_1q_2\equiv 5\pmod 8 \text{ or } \left(\frac{q_1q_2}{s}\right)=-1 , \\
			& \delta  rst&  \text{ whenever }  q_1q_2\equiv 1\pmod 8 \text{ and }  \left(\frac{q_1q_2}{r}\right)= \left(\frac{q_1q_2}{s}\right)=\left(\frac{q_1q_2}{t}\right)=-1  ,\\
			& \delta  rst&  \text{ whenever }  q_1q_2\equiv 5\pmod 8 \text{ and }     \left(\frac{q_1q_2}{s}\right)=\left(\frac{q_1q_2}{t}\right)=-1,\\
			&\delta rstk& \text{ whenever }  q_1q_2\equiv 5\pmod 8 \text{ and }  \left(\frac{q_1q_2}{r}\right)=\left(\frac{q_1q_2}{s}\right)=\left(\frac{q_1q_2}{t}\right)=\left(\frac{q_1q_2}{k}\right)=-1,
		\end{eqnarray*}
		where $\delta\in \{1,q_1,q_2\}$ and $r$, $s$, $t$ and   $k$   are prime numbers all different of $q_1$ and $q_2$.
		\begin{enumerate}[$\blacktriangleright$]

			\item 	  Let  $d=\delta r$.   
			We have:
			\begin{enumerate}[$\star$]
				\item If $q_1q_2\equiv 1\pmod 8$ and  $ \left(\frac{q_1q_2}{r}\right)=-1$, then $ \2r(A(K))=3-1-e_{K/F}$ and so      \cite[Theorems 3.4]{Azmouh2-rank}  gives
				$ \2r(A(K))=1$.

				\item If $q_1q_2\equiv 1\pmod 8$ and  $ \left(\frac{q_1q_2}{r}\right)=1$, then by Lemma \ref{corrlemma} we have    $ \2r(A(K))\in \{1,2  \}$. More precisely,  $ \2r(A(K))=1$ if and only if $\left(\frac{q_1 }{r}\right)=-1$.



				\item If $q_1q_2\equiv 5\pmod 8$ and  $ \left(\frac{q_1q_2}{r}\right)=-1$, then  $ \2r(A(K))=2-1-e_{K/F}$ and \cite[Theorems 3.3]{Azmouh2-rank} implies that    $ \2r(A(K))=1$.

				\item If $q_1q_2\equiv 5\pmod 8$ and  $ \left(\frac{q_1q_2}{r}\right)=1$, then $ \2r(A(K))=3-1-e_{K/F}\in \{1,2\}$ and  \cite[Theorems 3.3]{Azmouh2-rank} implies
				$ e_{K/F}=0$ if and only if $\left(\frac{-1}{r}\right)=\left(\frac{q_1}{r}\right)=1,$  thus $\2r(A(K))=2$ if and only if $\left(\frac{-1}{r}\right)=\left(\frac{q_1}{r}\right)=1.$ 
				
			\end{enumerate}
			\item  Let  $d=  \delta rs$ with   $q_1q_2\equiv 5\pmod 8 \text{ or } \left(\frac{q_1q_2}{s}\right)=-1$. 
			\begin{enumerate}[$\star$]
				\item Assume that $q_1q_2\equiv 1\pmod 8$. We have  : 
				\begin{enumerate}[$i)$]
					\item if $\left(\frac{q_1q_2}{r}\right)=\left(\frac{q_1q_2}{s}\right)=-1$, then \cite[Theorems 3.4]{Azmouh2-rank} gives $\2r(A(K))=4-1-e_{K/F}=2$.  
					\item if $\left(\frac{q_1q_2}{r}\right)=-\left(\frac{q_1q_2}{s}\right)=1$, then $\2r(A(K))=5-1-e_{K/F}=4-e_{K/F}$. By Lemma \ref{corrlemma} $e_{K/F}\in \{1,2\}$, more precisely, $e_{K/F}=2$ if and only if $\left(\frac{q_1}{r}\right)=- 1$.
					

				\end{enumerate}

				\item Assume that $q_1q_2\equiv 5\pmod 8$. We have:
				\begin{enumerate}[$i)$]
					\item If $\left(\frac{q_1q_2}{r}\right)=\left(\frac{q_1q_2}{s}\right)=-1$, then \cite[Theorems 3.3]{Azmouh2-rank} gives $\2r(A(K))=3-1-e_{K/F}=2$.
					\item If $\left(\frac{q_1q_2}{r}\right)=-\left(\frac{q_1q_2}{s}\right)=1$, then $\2r(A(K))=4-1-e_{K/F}=3-e_{K/F}$ and by \cite[Theorems 3.3]{Azmouh2-rank} we have $e_{K/F}\in \{0,1\}$, more precisely, $e_{K/F}=1$ if and only if $r\equiv3\pmod4$ or $\left(\frac{q_1}{r}\right)=-1.$
					\item If $\left(\frac{q_1q_2}{r}\right)=\left(\frac{q_1q_2}
					{s}\right)=1$, then  \cite[Theorems 3.3]{Azmouh2-rank}, $\2r(A(K))=5-1-e_{K/F}=4-e_{K/F}\in \{4,3,2\}$. More precisely,  $\2r(A(K))=2$ if and only if [$r\equiv 3\pmod4$ and $\left(\frac{q_1}{r}\right)=-\left(\frac{q_1}{s}\right)=1$] or [$r\equiv 3\pmod4$,  $s\equiv1\pmod4$and $\left(\frac{q_1}{s}\right)=-1$]. 
				\end{enumerate}
				
			\end{enumerate}

			\item Let 	$d=  \delta rst$ with $q_1q_2\equiv 1\pmod 8 \text{ and }  \left(\frac{q_1q_2}{r}\right)= \left(\frac{q_1q_2}{s}\right)=\left(\frac{q_1q_2}{t}\right)=-1 $. 
			In this case, \cite[Theorems 3.4]{Azmouh2-rank} gives $\2r(A(K))=5-1-e_{K/F}=3 $.
			\item Let 	$d=  \delta rst$ with $q_1q_2\equiv 5\pmod 8 \text{ and }  \left(\frac{q_1q_2}{s}\right)= \left(\frac{q_1q_2}{t}\right)=-1 $. We have:
			\begin{enumerate}[$\star$]
				\item If $\left(\frac{q_1q_2}{r}\right)=-1$, then   \cite[Theorems 3.3]{Azmouh2-rank} gives  $e_{K/F}=0 $ and so $\2r(A(K))=3 $.
				\item If $\left(\frac{q_1q_2}{r}\right)=1$, then $\2r(A(K))=5-1-e_{K/F}=4-e_{K/F}\in \{3,4\} $. More precisely, $\2r(A(K))=4$ if and only if $r\equiv1\pmod4$ and $\left(\frac{q_1}{r}\right)=1$.
				
			\end{enumerate} 
			
			\item Let 	$d=  \delta rstk$ with $q_1q_2\equiv 5\pmod 8 \text{ and }  \left(\frac{q_1q_2}{r}\right)=\left(\frac{q_1q_2}{s}\right)= \left(\frac{q_1q_2}{t}\right)=\left(\frac{q_1q_2}{k}\right)=-1 $.
			By \cite[Theorems 3.3]{Azmouh2-rank}we have $\2r(A(K))=5-1-e_{K/F}=4$.
			
		\end{enumerate}

	\end{proof}

	Let us recall that $L$ is named to be   any   real biquadratic fields  of the forms 
	$$L:=\QQ(\sqrt{ q_1q_2},\sqrt{r})\text{ or }\QQ(\sqrt{ q_1q_2},\sqrt{rs}) \text{ where } r\equiv 1\pmod 8\text{ and }  \left(\frac{q_1q_2}{r}\right)=1. $$
	  The following lemma that gives the $2$-rank of the class group of some real triquadratic number fields.

	\begin{lemma}\label{realTruquad2-rankq1q1} Let $K= \QQ(\sqrt{q_1q_2},\sqrt{d})\not=L$ and  $K_1= \QQ(\sqrt{q_1q_2},\sqrt{d}, \sqrt{2})$  where $q_1\equiv7 \pmod 8$,  $q_2\equiv3 \pmod 8$ are two prime numbers and \textcolor{blue}{$d \equiv 3\pmod  4$ } is a positive square free integer that is not divisible by  $q_1q_2$.
		
		\noindent $\mathbf{1)}$	Assume that $K_1= \QQ(\sqrt{q_1q_2},\sqrt{d}, \sqrt{2})$,  where   $d=\delta r\equiv 3\pmod 4 $ with $r$ and $s$ are two prime  numbers and    $\delta\in \{1,q_1,q_2\}$.  
		We have  $$\rg(A(K_1))\in\{1,2\}.$$ More precisely:
		$	\rg(A(K_1)) =1$ if and only if    we have   one of the following conditions :  
		\begin{enumerate}[\indent$\bullet$]
			\item $r\equiv 3\pmod 8$,
			\item $r\equiv 5\pmod 8$, $q_1\equiv 7\pmod 8$ and  $\left(\frac{q_1}{ r}\right)=-1$.

			
		\end{enumerate}

		\noindent $\mathbf{2)}$ 	Assume that $K_1= \QQ(\sqrt{q_1q_2},\sqrt{d}, \sqrt{2})$ with $d=\delta rs\equiv 3\pmod 4$, where $r$ and $s$ are two prime numbers such that  $\rg(A(K)) =1$.
		We have:  $$\rg(A(K_1))\in\{2,3,4,5\}.$$
		More precisely,  we have:

		\noindent$\rg(A(K_1))=2$ if and only if, after a suitable permutation of $r$ and $s$, we have one of the following conditions:
		
		\begin{enumerate}[$\C1:$]
			\item $ \left(\frac{q_1q_2}{s}\right)=\left(\frac{q_1q_2}{r}\right)=-1$ and we have one of the following congruence conditions:
			\begin{enumerate}[$\bullet$]
				\item  $r\equiv 5\pmod 8$ and $s\equiv 3\pmod 8$ with $q_1\equiv 7\pmod 8$ and $\left(\frac{q_1}{r}\right)=-1$,
				
				\item $r\equiv 3\pmod 8$ and $s\equiv 3\pmod 8$ with  $q_1\equiv 7\pmod 8$ and $\left(\frac{q_1}{r}\right)\not=\left(\frac{q_1}{s}\right)$. 
				

			\end{enumerate}

			\item	$ \left(\frac{q_1q_2}{r}\right)=-\left(\frac{q_1q_2}{s}\right)=-1$ and we have one of the following congruence conditions:
			\begin{enumerate}[$\bullet$]
				\item  $ r\equiv         s\equiv 3\pmod 8$ with $q_1\equiv 7\pmod 8$ and $\left(\frac{q_1}{r}\right) \not=\left(\frac{q_1}{s}\right)$,
				\item $ r\equiv     3\pmod 8$ and  $    s\equiv 5\pmod 8$  with $q_1\equiv 7\pmod 8$ and $\left(\frac{q_1}{s}\right) =-1$,  
				
				\item   $ r\equiv     5\pmod 8$ and  $    s\equiv 3\pmod 8$ with      $q_1\equiv 7\pmod 8$ and $\left(\frac{q_1}{r}\right) =-1$.
			\end{enumerate}

			\item	$ \left(\frac{q_1q_2}{r}\right)=\left(\frac{q_1q_2}{s}\right)=1$ and we have one of the following congruence conditions:
			\begin{enumerate}[$\bullet$]
				\item  $r\equiv 3\pmod 8$ and $s\equiv 3\pmod 8$  with       $ q_1\equiv 7\pmod 8$ and $\left(\frac{q_1}{r}\right) \not=\left(\frac{q_1}{s}\right) $,
				
				\item   $r\equiv 5\pmod 8$ and $s\equiv 3\pmod 8$  with       $ q_1\equiv 7\pmod 8$ and $\left(\frac{q_1}{r}\right) =-1$.
			\end{enumerate}

		\end{enumerate}

		\noindent $\rg(A(K_1))=3$ if and only if, after a suitable permutation of $r$ and $s$, we have one of the following conditions:
		\begin{enumerate}[$\C1:$]
			\item $ \left(\frac{q_1q_2}{s}\right)=\left(\frac{q_1q_2}{r}\right)=-1$ and we have one of the following congruence conditions:
			\begin{enumerate}[$\bullet$]
				\item    $r\equiv 5\pmod 8$ and $s\equiv 3\pmod 8$ with      $ q_1\equiv 7\pmod 8$ and $\left(\frac{q_1}{r}\right) =1$,

				\item  $r\equiv 5\pmod 8$ and $s\equiv 5\pmod 8$  with    $q_1\equiv 7\pmod 8$ and $[\left(\frac{q_1}{r}\right) =-1$ or $\left(\frac{q_1}{s}\right) =-1]$,
				
				\item $r\equiv 3\pmod 8$ and $s\equiv 3\pmod 8$ with      $q_1\equiv 7\pmod 8$ and $\left(\frac{q_1}{r}\right) =\left(\frac{q_1}{s}\right)  $,

				\item  $r\equiv 3\pmod 8$ and $s\equiv 7\pmod 8$ with    $ q_1\equiv 7\pmod 8$,
				
				\item  $r\equiv 5\pmod 8$ and $s\equiv 7\pmod 8$ with   $q_1\equiv 7\pmod 8$ and $\left(\frac{q_1}{r}\right) =-1$,

				\item $r\equiv 5\pmod 8$ and $s\equiv 1\pmod 8$ with    $q_1\equiv 7\pmod 8$ and $\left(\frac{q_1}{r}\right) =-1$,

				\item $r\equiv 3\pmod 8$ and $s\equiv 1\pmod 8$.
				
			\end{enumerate}

			\item	$ \left(\frac{q_1q_2}{r}\right)=-\left(\frac{q_1q_2}{s}\right)=-1$ and we have one of the following congruence conditions:
			\begin{enumerate}[$\bullet$]
				
				\item  $r\equiv   s\equiv 5\pmod 8$ with     $q_1\equiv 7\pmod 8$ and $[\left(\frac{q_1}{r}\right) =-1$ or $ \left(\frac{q_1}{s}\right)=-1  ]$,
				
				\item $r\equiv   s\equiv 3\pmod 8$ with    $q_1\equiv 7\pmod 8$ and  $\left(\frac{q_1}{r}\right) =\left(\frac{q_1}{s}\right)$,
				
				\item $r\equiv     3\pmod 8$ and  $   s\equiv 5\pmod 8$ with   $q_1\equiv 7\pmod 8$ and    $\left(\frac{q_1}{s}\right)=1  $,

				\item $ r\equiv     5\pmod 8$ and  $    s\equiv 3\pmod 8$ with     $ q_1\equiv 7\pmod 8$ and $\left(\frac{q_1}{r}\right) =1$,

				\item       $ r\equiv     1\pmod 8$ and   $    s\equiv 5\pmod 8$    with    $q_1\equiv 7\pmod 8$ and $\left(\frac{q_1}{s}\right) =-1$.
				

			\end{enumerate}

			\item	$ \left(\frac{q_1q_2}{r}\right)=\left(\frac{q_1q_2}{s}\right)=1$ and we have one of the following congruence conditions:
			\begin{enumerate}[$\bullet$]
				\item   $r\equiv 7\pmod 8$ and $s\equiv 5 \pmod 8$ with       $q_1\equiv 7\pmod 8$ and $\left(\frac{q_1}{s}\right) =-1$,

				\item $r\equiv 7\pmod 8$ and $s\equiv 3 \pmod 8$ with      $ q_1\equiv 7\pmod 8$ and $\left(\frac{q_1}{r}\right) \not=\left(\frac{q_1}{s}\right)$,
				
				\item $r\equiv 3\pmod 8$ and $s\equiv 3\pmod 8$ with       $q_1\equiv 7\pmod 8$ and $\left(\frac{q_1}{r}\right) =\left(\frac{q_1}{s}\right) $,

				\item $r\equiv 5\pmod 8$ and $s\equiv 5\pmod 8$ with       $ q_1\equiv 7\pmod 8$ and $[\left(\frac{q_1}{r}\right)=-1$ or $  \left(\frac{q_1}{s}\right)=-1 ]$,
				
				\item $r\equiv 5\pmod 8$ and $s\equiv 3\pmod 8$ with      $ q_1\equiv 7\pmod 8$ and $\left(\frac{q_1}{r}\right)=1$.
			\end{enumerate}

		\end{enumerate}

		\noindent $\mathbf{3)}$ 	Assume that $K_1= \QQ(\sqrt{q_1q_2},\sqrt{d}, \sqrt{2})$ with $d=\delta rst\equiv 1\pmod 4$, where $r$, $s$ and $t$ are two prime numbers such that
		$ \left(\frac{q_1q_2}{r}\right)=\left(\frac{q_1q_2}{s}\right)=\left(\frac{q_1q_2}{t}\right)=-1$ or $ \left(\frac{q_1q_2}{r}\right)=-\left(\frac{q_1q_2}{s}\right)=\left(\frac{q_1q_2}{t}\right)=-1$. Then, we have: 
		$$\rg(A(K_1))\in\{4,5,6\}.$$

	\end{lemma}
	\begin{proof}
		Let $d'=\frac{1}{q_1}d$ or $q_1d$ according to whether $q_1$ devises $d$ or not. Notice that $d'$ is an integers such that $d'\equiv1\pmod 4$. Let $K_1'=\QQ(\sqrt{q_1q_2},\sqrt{2},\sqrt{d'}) $ and $F=\QQ(\sqrt{q_1q_2},\sqrt{2}) $. Then $\rg(K_1)=t_{K_1/F}-1-e_{K_1/F}$. We have 	$E_{F}=\langle-1, \vep_{2},\vep_{q_1q_2}, \sqrt{\vep_{2q_1q_2}\vep_{q_1q_2}}  \rangle $  is the unit group of $F$.
		Let $\vep\in \{-1, \vep_{2},\vep_{q_1q_2}, \sqrt{\vep_{2q_1q_2}\vep_{q_1q_2}}\}$.

			Let $\ell \not\in\{q_1,q_2\}$ (resp. $r\not\in\{q_1,q_2\}  $) denotes an odd prime divisor  of $d$ that is not satisfying (resp. that is satisfying)   the conditions $[\ell\equiv 7\pmod 8$ and  $\left(\frac{q_1q_2}{ \ell}\right)=1]$ (resp. $[r\equiv 7\pmod 8$ and  $\left(\frac{q_1q_2}{ r}\right)=1]$).	 So    $\ell$ (resp. $r$) decomposes into two (resp. four) prime ideals  $\mathcal{L}$ and $\mathcal{L}'$ (resp. $\mathcal{R}_1$, $\mathcal{R}_2$, $\mathcal{R}_1'$ and $\mathcal{R}_2'$) of $F$, with $\mathcal{R}_1$ (resp. $\mathcal{R}_1'$) and  $\mathcal{R}_2$ (resp. $\mathcal{R}_2'$) are conjugate in the extension $F/\QQ(\sqrt{2})$.
			Thus, by the product formula of norm residue symbols, we have
			$$\left(\frac{\vep ,\, d}{ \mathfrak 2}\right)\prod_{\ell|d} \left(\frac{\vep ,\, d}{ \mathcal L}\right)\left(\frac{\vep ,\, d}{ \mathcal L'}\right)\prod_{r|d} \left(\frac{\vep ,\, d}{ \mathcal R_1}\right) \left(\frac{\vep ,\, d}{ \mathcal R_2}\right) \left(\frac{\vep ,\, d}{ \mathcal R_1'}\right) \left(\frac{\vep ,\, d}{ \mathcal R_2'}\right)=1,$$
			Since $\left(\frac{\vep ,\, d}{ \mathcal L}\right)=\left(\frac{\vep ,\, d}{ \mathcal L'}\right)$,
			$\left(\frac{\vep ,\, d}{ \mathcal R_1}\right) =\left(\frac{\vep ,\, d}{ \mathcal R_2}\right)$
			and $\left(\frac{\vep ,\, d}{ \mathcal R_1'}\right) =\left(\frac{\vep ,\, d}{ \mathcal R_2'}\right)$ (cf. The proof of \cite[Lemma 3.2]{ChemseddinGreenbergConjectureI}), 
			this implies that $\left(\frac{\vep ,\, d}{ \mathfrak 2}\right)=1$. Therefore, $\rg(A(K_1))=t_{K_1/F}-1-e_{K_1/F}=(t_{K_1'/F}+1)-1-e_{K_1'/F}=\rg(A(K_1'))+1$.
		
		Hence, the result follows from	    \cite[Lemma 3.2]{ChemseddinGreenbergConjectureI}.
	\end{proof}

	\begin{remark} Let $K$ be a real biquadratic field of the form $D)$ with  $K\not= L$.
		By combining  Lemma   \ref{realTruquad2-rankq1q1},  Lemma \ref{lemma2qC},  Lemma \ref{lemmaq1q2Bd=3mod4}      and  Lemma \ref{lm fukuda}, we get that $ \rg(A(K_\infty))\leq 2$ and $ \rg(A(K_\infty))=\rg( A(K))$ if and only if $K$ is taking one of the forms in  the items $1)$, $2)$,  $3)$,  $4)$, $5)$, $6)$, $7)$ and   $9)$ of the main first theorem. 
	\end{remark}

	\begin{lemma}\label{castrivial1}
		Let  $p\equiv1 \pmod 4$   and $d\equiv 1 \pmod 4$ are two prime numbers such that $\left(\frac{p}{d}\right)=-1$ or $[\left(\frac{p}{d}\right)=1$ and $\left(\frac{p}{d}\right)_4\not=\left(\frac{d}{p}\right)_4]$. Put  $K_1=\QQ(\sqrt{p}, \sqrt{d}, \sqrt{2})$, Then $h_2(K_1)$ is odd if and only if 
		$h_2(2pd)=4$ and $q(k_1)=1$, where $k_1=\QQ(\sqrt{pd}, \sqrt{2})$.
	\end{lemma}
	\begin{proof}
		Notice that $K_1$ is a quartic unramified extension of $k=\QQ(\sqrt{2pd})$ and  $\2r(A(k))=2$ (cf. \cite[pp. 314-315]{kaplan76}). Thus, by class field  theory, $h(K_1)$ is odd if and only if $h_2(2pd)=4$ and   $k^{(1)}=k^{(2)}$, i.e. the Hilbert $2$-class field tower of $k$ stops at the first layer. On the other hand, we have
		$h_2(k_1)=\frac{1}{4}q(k_1) h_2(2) h_2(pd)h_2({2pd})=\frac{1}{2}q(k)\cdot h_2({k})$ $(cf.$  \cite[Theorem 1]{kuvcera1995parity}$).$
		As $k_1$ is an unramified quadratic extension of $k$, then by  Proposition \ref{LemBenjShn}, $k^{(1)}=k^{(2)}$ if and only if 
		$h_2(k_1)=\frac12 h_2(k)$. Thus, $h_2(K_1)=1$ if and only if $\frac{1}{2}q(k_1)\cdot h_2({2pd})=\frac12 h_2(2pd)$ and $h_2(2pd)=4$, which is equivalent to 
		$q(k_1)=1$ and $h_2(2pd)=4$.
	\end{proof}


	\begin{lemma}\label{castrivialq1q2}
		Let $K=\QQ(\sqrt{p},\sqrt{q_1q_2},\sqrt{2})$,    where $p$ and $q_i$ are three primes such that $p\equiv1 \pmod 4$ and $q_1\equiv q_2 \equiv 7 \pmod 8$ and $\left(\frac{q_2}{ p}\right)=-1$. Then the class number of $K_1$ is even. 
	\end{lemma}
	\begin{proof}
		Let $k=\QQ(\sqrt{p},\sqrt{2q_1q_2})$. As $h_2(p) =1$, $h_2(2q_1q_2)\geq 4$ (cf. \cite[p. 315]{kaplan76}) and  $h_2(2pq_1q_2)\geq 4$ (by genus theory) , then 
		$h_2(k)=\frac{1}{4}q(k) h_2(p) h_2(2q_1q_2)h_2({2pq_1q_2})\geq 4 .$ 
		Thus, the class number of $K_1$ is even, in fact, it is an unramified quadratic extension of $k$.
	\end{proof}	 
	
	\begin{remark}\label{rempq1q1} 	Let $K_1=\QQ(\sqrt{p},\sqrt{q_1q_2},\sqrt{2})$,    where $p$, $q_1$ and $q_3$ are three prime numbers such that $p\equiv1 \pmod 4$, $q_1\equiv  3 \pmod 4$, $q_2 \equiv3\pmod8$ and $\left(\frac{q_2}{ p}\right)=-1$. By \cite[Lemma 3.12]{ChemseddinGreenbergConjectureI}, $h_2(K_1)=1$ 
		if and only if we have one of the following conditions:
		\begin{enumerate}[$\C1:$]
			\item $p\equiv5\pmod8, q_1\equiv3\pmod8 $ and $\left(\frac{q_1}{p}\right)=1$,
			\item $p\equiv5\pmod8, q_1\equiv7\pmod8 $ and $\left(\frac{q_1}{p}\right)=-1$.
		\end{enumerate}
		
	\end{remark}
	
	\begin{remark}\label{rempd} 	Let $K_1=\QQ(\sqrt{p},\sqrt{d},\sqrt{2})$,    where $p\equiv1 \pmod 4$   and $d\equiv 3 \pmod 4$ are two prime numbers such that $\left(\frac{p}{d}\right)=-1$ or $p\equiv5 \pmod 8$. By \cite[Lemma 3.5-$1)$]{ChemseddinGreenbergConjectureI}, $h_2(K_1)=1$ 
		if and only if  $p\equiv5 \pmod 8$.
	\end{remark}

	\begin{remark} Let $K$ be a real biquadratic field of the form $F)$.
		By combining    Lemma   \ref{castrivial1},  Lemma \ref{castrivialq1q2}, Remark \ref{rempq1q1}, Remark   \ref{rempd} and  Lemma \ref{lm fukuda}, we get that $ \rg(A(K_\infty))\leq 2$ and $ \rg(A(K_\infty))=\rg( A(K))$ if and only if $K$ is taking one of the forms in  the items $10)$, $11)$,  $12)$  of the main theorem. 
	\end{remark}
	
	This completes the proof of the first main  theorem.

	\section{\bf Some infinite families of real biquadratic fields  $K$ such that $\rg(A(K))=\rg(A_\infty(K))=3$}\label{Sec.2}	
	As the previous investigations and \cite{ChemseddinGreenbergConjectureI} give  long lists of real biquadratic fields $K$ such that  $\rg(A(K))=\rg(A_\infty(K))$ and $\rg(A_\infty(K))\leq 2$, we show  that the stability of the $2$-rank of the class group in the  cyclotomic $\ZZ_2$-extension    may be valid for larger ranks. In fact, we give some infinite families of real biquadratic fields $K$ such that 
	$\rg(A(K))=\rg(A_\infty(K))=3$. Let us start with the following lemma.
	
	\begin{lemma}\label{lemmaaaa} Let $K=\QQ(\sqrt{q_1q_2},\sqrt{2d})$ or $\QQ(\sqrt{q_1q_2},\sqrt{d})$, where $q_1\equiv7 \pmod 8$,  $q_2\equiv3 \pmod 8$ are two prime numbers  and    $d\equiv 3 \pmod 4$ is a  positive square-free integer  that is not divisible by   $q_1q_2$.Then $\2r(K)=3$ if and only if $d$ takes one of the following forms :

		\begin{enumerate}[$1)$]

			\item Let $d=\delta rs$, where $\delta\in \{1,q_1,q_2\}$, $r$ and $s$ are prime numbers such that $ \left(\frac{q_1q_2}{r}\right)=-\left(\frac{q_1q_2}{s}\right)=-1$, $ s \equiv 1\pmod4 \text{ and }\left(\frac{q_1}{s}\right)=1.$

			\item Let $d=\delta rs$, where $\delta\in \{1,q_1,q_2\}$, $r$ and $s$ are prime numbers such that $ \left(\frac{q_1q_2}{s}\right)=\left(\frac{q_1q_2}{r}\right)=1$ and  one of 
			the following conditions holds  \begin{enumerate}[$\bullet$]
				\item $ r\equiv     1\pmod 4$, $\left(\frac{q_1}{r}\right) =1$ and $\left(\frac{q_1}{s}\right) =-1$.
				\item $ r\equiv     3\pmod 4$ and $\left(\frac{q_1}{r}\right) =\left(\frac{q_1}{s}\right)=1$.
				\item $ r\equiv   s \equiv   1\pmod 4$ and $\left(\frac{q_1}{r}\right) =-1$.
				\item $ r\equiv   s \equiv   3\pmod 4$ and $\left(\frac{q_1}{r}\right)\cdot \left(\frac{q_1}{s}\right) =1$.
				\item $ r\equiv     3\pmod 4$, $ s\equiv 1\pmod 4$ and $\left(\frac{q_1}{s}\right) =1$.
			\end{enumerate}
			
			\item Let $d=\delta rst$, where $\delta\in \{1,q_1,q_2\}$, $r$, $s$ and $t$ are prime numbers such that  $ \left(\frac{q_1q_2}{r}\right)=  \left(\frac{q_1q_2}{s}\right)=-\left(\frac{q_1q_2}{t}\right)=1$ and  one of 
			the following conditions holds \begin{enumerate}[$\bullet$]
				\item $ r\equiv 3\pmod 4$, $\left(\frac{q_1}{s}\right) =-1$ and $\left(\frac{q_1}{r}\right) =1$.
				\item $ r\equiv  3\pmod 4$, $ s\equiv  1\pmod 4$ and  $\left(\frac{q_1}{s}\right) =-1$.
			\end{enumerate}
			\item Let $d=\delta rst$, where $\delta\in \{1,q_1,q_2\}$, $r$, $s$ and $t$ are prime numbers such that   $ \left(\frac{q_1q_2}{r}\right)=-  \left(\frac{q_1q_2}{s}\right)=\left(\frac{q_1q_2}{t}\right)=-1$, and  one of 
			the following conditions holds \begin{enumerate}[$\bullet$]
				\item $ s\equiv  3\pmod 4$.
				\item $\left(\frac{q_1}{s}\right) =-1$.
			\end{enumerate} 
			\item Let $d=\delta rst$, where $\delta\in \{1,q_1,q_2\}$, $r$, $s$ and $t$ are prime numbers such that and  $ \left(\frac{q_1q_2}{r}\right)= \left(\frac{q_1q_2}{s}\right)=\left(\frac{q_1q_2}{t}\right)=-1$.

		\end{enumerate}
		
	\end{lemma}
	\begin{proof}
		
		Assume that $K=\QQ(\sqrt{q_1q_2},\sqrt{2d})$ where   $q_1\equiv7 \pmod 8$,  $q_2\equiv3 \pmod 8$ are two prime numbers  and {\color{red}$d\equiv 3\pmod 4$} is any       positive   square-free integer that is not divisible by   $q_1q_2$.
		Let $t_{K/F}$ be the number of prime ideals of $ F=\QQ(\sqrt{q_1q_2})$ that are ramified in $K$.  Since we have $q_1q_2\equiv5 \pmod 8$ so there is only one ideal of $F$ lying above $2$ that ramified in $K$.
		Thus,  $ \2r(A(K))=t_{K/F}-1-e_{K/F}$. 
		As  $e_{K/F}\in\{0,1,2\}$, then the inequality $\2r(A(K))\leq 3$ implies that $t_{K/F}\leq 6$. Therefore, $d$ takes one of the following five forms
		\begin{eqnarray*}
			&\delta r, &  \\
			& \delta  rs,&  \\
			& \delta  rst&  \text{ whenever }      \left(\frac{q_1q_2}{t}\right)=-1,\\
			&\delta rstu& \text{ whenever }    \left(\frac{q_1q_2}{s}\right)=\left(\frac{q_1q_2}{t}\right)=\left(\frac{q_1q_2}{u}\right)=-1,\\
			&\delta rstuv& \text{ whenever }    \left(\frac{q_1q_2}{r}\right)=\left(\frac{q_1q_2}{s}\right)=\left(\frac{q_1q_2}{t}\right)=\left(\frac{q_1q_2}{u}\right)=\left(\frac{q_1q_2}{v}\right)=-1,
		\end{eqnarray*}
		where $\delta\in \{1,q_1,q_2\}$ and $r$, $s$, $t$, $u$, and   $v$   are prime numbers all different of $q_1$ and $q_2$.
		
		\begin{enumerate}[$\blacktriangleright$]

			\item 	  Let  $d=\delta r$.   If $\left(\frac{q_1q_2}{r}\right)=-1$, then $ \2r(A(K))=2-1-e_{K/F}=1-e=1$  by    \cite[Theorems 3.3]{Azmouh2-rank}. In case $\left(\frac{q_1q_2}{r}\right)=1$ we get $\2r(A(K))=3-1-e_{K/F}=2-e\leq 2.$
			
			\item Let  $d=\delta rs$, we have the following sub-cases :\begin{enumerate}[$\star$]
				\item If $\left(\frac{q_1q_2}{r}\right)=\left(\frac{q_1q_2}{s}\right)=1$, then $ \2r(A(K))=5-1-e_{K/F}=4-e\in \{4,3,2\}$ more precisely by    \cite[Theorems 3.3]{Azmouh2-rank} we have   $ \2r(A(K))=4$ if and only if $r\equiv s \equiv 1 \pmod 4$ and $\left(\frac{q_1}{r}\right)=\left(\frac{q_1}{s}\right)=1$, and $ \2r(A(K))=2$ if and only if [$r\equiv 3 \pmod 4$, $s\equiv 1 \pmod 4$ and $\left(\frac{q_1}{s}\right)=1$ or $r\equiv 3 \pmod 4$, $\left(\frac{q_1}{s}\right)=-1$ and $\left(\frac{q_1}{r}\right)=1$].

				\item If $\left(\frac{q_1q_2}{r}\right)=-\left(\frac{q_1q_2}{s}\right)=-1$, then $ \2r(A(K))=4-1-e_{K/F}\in \{3,2\}$ more precisely by    \cite[Theorems 3.3]{Azmouh2-rank} we have   $ \2r(A(K))=3$ if and only if $ s \equiv 1 \pmod 4$ and $\left(\frac{q_1}{s}\right)=1$.
				
				\item If $\left(\frac{q_1q_2}{r}\right)=\left(\frac{q_1q_2}{s}\right)=-1$, then $ \2r(A(K))=3-1-e_{K/F}=2$ by    \cite[Theorems 3.3]{Azmouh2-rank}.

			\end{enumerate}
			\item Let  $d=\delta rst$  and $\left(\frac{q_1q_2}{t}\right)=-1$, we have the following sub-cases :\begin{enumerate}[$\star$]
				\item If $\left(\frac{q_1q_2}{r}\right)=\left(\frac{q_1q_2}{s}\right)=1$, then $ \2r(A(K))=6-1-e_{K/F}\in \{5,4,3\}$ more precisely by    \cite[Theorems 3.3]{Azmouh2-rank}  we have   $ \2r(A(K))=3$ if and only if [$r\equiv 3 \pmod 4$,  $\left(\frac{q_1}{s}\right)=-1$ and $\left(\frac{q_1}{r}\right)=1$]  or [$r\equiv 3 \pmod 4$, $s\equiv 1 \pmod 4$ and $\left(\frac{q_1}{s}\right)=-1$].

				\item If $\left(\frac{q_1q_2}{r}\right)=-\left(\frac{q_1q_2}{s}\right)=-1$, then $ \2r(A(K))=5-1-e_{K/F}\in \{4,3\}$ more precisely by    \cite[Theorems 3.3]{Azmouh2-rank} we have   $ \2r(A(K))=3$ if and only if $ s \equiv 3 \pmod 4$ or $\left(\frac{q_1}{s}\right)=-1$.
				
				\item If $\left(\frac{q_1q_2}{r}\right)=\left(\frac{q_1q_2}{s}\right)=-1$, then $ \2r(A(K))=4-1-e_{K/F}=3$ by    \cite[Theorems 3.3]{Azmouh2-rank}.

			\end{enumerate} 
			
			 We proceed  similarly the remaining cases and check that    $ \2r(A(K))\geq 4 $.
			
		 	\end{enumerate}
	 	\end{proof}

	So we have the following    proposition which is a deduction of Lemma \ref{lemmaaaa},   Lemma \ref{realTruquad2-rankq1q1}-$2)$ and Fukuda's result.  
	
	
	\begin{proposition} \label{secondmaintheorem}
		Let $K$ be a real biquadratic number field that is in one of the following forms.  
		
		\begin{enumerate}[$1)$]
			\item $K=\QQ(\sqrt{q_1q_2},\sqrt{\delta rs})$ or $K=\QQ(\sqrt{q_1q_2},\sqrt{2\delta rs})$, where  $q_1\equiv 7\pmod 8$ and $q_2\equiv 3\pmod 8$, $r$ and $s$ are prime numbers, $\delta\in\{1,q_1,q_2\}$ is such that 
			$\delta rs\equiv 3\pmod4$,	$ \left(\frac{q_1q_2}{s}\right)=\left(\frac{q_1q_2}{r}\right)=1$ and we have one of the following congruence conditions:

			\begin{enumerate}[$\bullet$]
				\item  $r\equiv s\equiv 3\pmod 8$  and  $\left(\frac{q_1}{r}\right)=\left(\frac{q_1}{s}\right)$,
				\item  $r\equiv s\equiv 5\pmod 8$  and    $\left(\frac{q_1}{r}\right)=-1$,
				\item $r\equiv  5\pmod 8$, $s\equiv  3\pmod 8$, $\left(\frac{q_1}{r}\right)=1$.
			\end{enumerate}

			\item $K=\QQ(\sqrt{q_1q_2},\sqrt{\delta rs})$ or $K=\QQ(\sqrt{q_1q_2},\sqrt{2\delta rs})$, where  $q_1\equiv 7\pmod 8$ and $q_2\equiv 3\pmod 8$, $r$ and $s$ are prime numbers, $\delta\in\{1,q_1,q_2\}$ is such that 
			$\delta rs\equiv 3\pmod4$	$ \left(\frac{q_1q_2}{r}\right)=-\left(\frac{q_1q_2}{s}\right)=-1$ and we have one of the following congruence conditions:

			\begin{enumerate}[$\bullet$]
				\item  $r\equiv s\equiv 5\pmod 8$  and  $\left(\frac{q_1}{s}\right)=-\left(\frac{q_1}{r}\right)=1$,
				\item  $r\equiv 3\pmod 8$, $s\equiv 5\pmod 8$  and    $\left(\frac{q_1}{s}\right)=1$.
			\end{enumerate}
		\end{enumerate}	
		{\bf Then     $ \rg(A(K_\infty))=\rg( A(K))$   and  $\rg(A(K_\infty))=3$.}
	\end{proposition}

	\section{\bf Investigation  of  Iwasawa module and Greenberg's conjecture}\label{Sec.3}
	
	In this section, we give some families of real biquadratic fields satisfying Greenberg's conjecture and specify the structure of their Iwasawa module. In particular, we determine the complete list of all real biquadratic fields with a trivial Iwasawa module.  We note that the analogue of this problem for real quadratic fields was the subject of the paper
		\cite{mouhib-mova}.

	\subsection{Triviality of Iwasawa module of real biquadratic fields}		
In this subsection we determine the list of all real biquadratic fields with trivial $2$-Iwasawa module.

	\begin{lemma}\label{lemmCaseE}
		Let $K$ be a real biquadratic field   of the form $E)$, that is $K=\QQ(\sqrt{q_1q_2},\sqrt{d})$ or $\QQ(\sqrt{q_1q_2},\sqrt{2d})$, where $q_1\equiv3 \pmod 8$,  $q_2\equiv3 \pmod 8$ are two prime numbers  and    $d\equiv 3 \pmod 4$ is a  positive square-free integer  that is not divisible by   $q_1q_2$.
		Then the class number of $K_1$ is even and $A(K_\infty)$ is {\bf not} trivial.
	\end{lemma}
	\begin{proof} Put $F=\QQ(\sqrt{q_1q_2})$.
		Notice that $K/ F$ is a	QO-extension.  Since we have two prime of $F$ above $2$, we have $\2r(K)=t_{K/F}-1-e_{K/F}\geq 3-1-e_{K/F}=2-e_{K/F}$. Notice that according to \cite[Theorem 3.4]{Azmouh2-rank} and Lemma \ref{corrlemma}, $e_{K/F}=2$ if and only if $t_{K/F}\geq 4$ (i.e. there is a prime divisor of $d$ that decomposes into two primes in $F$).
		Thus, $\2r(K) \geq 1$. Since $K_\infty /K$ is totally ramified, then  $A(K_\infty)$ is  not  trivial.
	\end{proof}
	
	\medskip
	Now we can state and prove the second main theorem of this paper.
	
	\begin{theorem}[{\bf The Second Main Theorem}]
		Let $K $ be any real biquadratic number field. Then $A(K_\infty)$ is trivial if and only if $K$ is in one of the following forms:
		\begin{enumerate}[$1)$]
			\item  $K=\QQ(\sqrt{\eta_1 q},\sqrt{\eta_2 r})$,     where $\eta_1, \eta_2\in\{1,2\}$,  $q\equiv 3\pmod 4$ and $r$ are two prime numbers  such that we have one of the following   conditions:
			\begin{enumerate}[$\C1:$]
				\item   $r\equiv 3$ or $5\pmod 8$,
				
				\item   $r\equiv 7\pmod 8$ and  $q\equiv 3\pmod 8$.
			\end{enumerate}

				%

			\item  Let $K=\QQ(\sqrt{\eta_1q_1q_2},\sqrt{\eta_2\delta r})$, where  $\eta_1, \eta_2\in\{1,2\}$,  $q_1\equiv3 \pmod 4$,  $q_2\equiv3 \pmod 8$ and $r$ are prime numbers, and   $\delta\in  \{1,q_1,q_2\}$ with $\delta r\equiv 1\pmod 4$ and such that we have one of the following conditions:
			\begin{enumerate}[$\C1:$]
				\item $r\equiv 3\pmod 8$,
				\item $r\equiv 5\pmod 8$, $ q_1\equiv 3\pmod 8$ and  $\left(\frac{q_1q_2}{ r}\right)=-1$,
				\item $r\equiv 5\pmod 8$, $ q_1\equiv 7\pmod 8$ and  $\left(\frac{q_1}{ r}\right)=-1$,

				\item  $r\equiv  7\pmod 8$,   $q_1\equiv 3\pmod 8$.
				
			\end{enumerate}

			\item 	$K=\QQ(\sqrt{p_1},\sqrt{p_2})$, where $p_1\equiv p_2 \equiv 1 \pmod4$ are two prime numbers satisfying the following conditions:
			\begin{enumerate}[$a)$]
				\item $\left(\frac{p_1}{p_2}\right)=-1$ or $[\left(\frac{p_1}{p_2}\right)=1$ and $\left(\frac{p_1}{p_2}\right)_4\not=\left(\frac{p_2}{p_1}\right)_4]$,
				\item     
				$h_2(2p_1p_2)=4$ and
				\item  $q(k_1)=1$, where $k_1=\QQ(\sqrt{p_1p_2}, \sqrt{2})$.
			\end{enumerate}

		\end{enumerate}
	\end{theorem}
	\begin{proof}Notice that for any real biquadratic field $K$, we have $K_\infty/K$ is totally ramified. So
		we need to distinguish the following two  cases:
		\begin{enumerate}[$1)$]
			\item  Assume that $K_1/K$ is unramified.  In this case,  $A(K_\infty)$ is trivial, implies that $A(K_1)$ is trivial.
			The fact that $K_1/K$ is unramified, implies that  $K$ is in one of the following forms:
			\begin{enumerate}[$\bullet$]
				\item
				$K=\QQ(\sqrt{d_1},\sqrt{2d_2})$ with $d_2\equiv 1\pmod 4$,
				\item  $K=\QQ(\sqrt{d_1},\sqrt{2d_2})$ with $d_1\equiv 3\pmod 4$ or 
				\item $K=\QQ(\sqrt{2d_1},\sqrt{2d_2})$ with $d_1\equiv 1\pmod 4$ and $d_2\equiv 3$ or $1\pmod 4$.
			\end{enumerate}
			Let $K'= \QQ(\sqrt{d_1},\sqrt{d_2})$ with $d_1$ and $d_2$ satisfy one of the above conditions. Notice that we have $K_\infty=K'_\infty$ (in fact, $K_1=K'_1)$ and $K_1/K'$ is ramified. Thus, $A(K)$ is trivial, implies that $K_1/K$ is an extension of $\mathrm{QO}$-fields. Therefore, it suffices to   list the form of  $K'$ from our first main theorem, Lemma \ref{lemmCaseE} and 
			\cite[The main Theorem]{ChemseddinGreenbergConjectureI}.
			\item  Assume that  $K_1/K$ is ramified. Here  we have two  subcases:
			\begin{enumerate}[$a)$]
				\item Assume that $K$ takes the form $E)$. In this case,  $A(K_\infty)$ is not trivial (cf. Lemma \ref{lemmCaseE}).
				
				\item    $K$  is not of the form $E)$. 	In this case, notice that if $A(K_\infty)$ is trivial, then $A(K)$ is trivial. Thus $K_1/K$ is an extension of $\mathrm{QO}$-fields. Therefore, the list of $K$ is given by  
				our first main theorem and 
				\cite[The main Theorem]{ChemseddinGreenbergConjectureI}.
			\end{enumerate}
			
		\end{enumerate}
		Plugging these investigations to our first main theorem and 
		\cite[The main Theorem]{ChemseddinGreenbergConjectureI}, we complete the proof.
	\end{proof}

	\subsection{Greenbergs conjecture over the fields $K_\nu=\mathbb{Q}(\sqrt{\nu q},\sqrt{rs})$, with  $\nu\in\{1,2\}$}		
	
	Let  $q\equiv    3\pmod 4$  and $r\equiv s\equiv     3\pmod 8$   be three distinct prime numbers such that $\left(\frac{q}{s}\right)=\left(\frac{q}{r}\right)=(-1)^{\delta_{\nu, 2}}$.
	Consider the fields	$K_\nu=\mathbb{Q}(\sqrt{\nu q},\sqrt{rs})$ and $F_\nu=\mathbb{Q}(\sqrt{\nu qrs})$,  with  $\nu\in\{1,2\}$. Let $a$   and $b$  be the integers such that
	$ \varepsilon_{\nu qrs}=a+b\sqrt{\nu  qrs}$.  	 Notice that 
	$K_\nu/F_\nu$ is an unramified quadratic extension. Moreover, we have $\rg(A(K_\nu))=\rg(A_\infty(K_{\nu}))=2$ (cf. \cite[Items $5)$ and $14)$ of Theorem 1.4]{ChemseddinGreenbergConjectureI}) and 
	$\rg(A(F_\nu))=\rg(A_\infty(F_{\nu}))=2$   (in fact according to 	\cite[p. 164]{BenjShn(22)}, we have $\rg(A(F_{\nu}))=2$). 
	Let $K_{\nu,n}$ and $F_{\nu,n}$ denote the $n$th layer of  cyclotomic $Z_2$-extension of $K_{\nu}$ and $F_{\nu}$ respectively. 
	In this subsection we aim to improve and extend   \cite[Theorem 4.11 and Corollary 4.12]{ChemseddinGreenbergConjectureI}. 
	For some works in this direction we refer the reader to 	\cite{AziziRezzouguiZekhniniPeriodica,AziziRezzouguiZekhniniDebrecen,CE-EM_2025}.

	\bigskip

	\begin{lemma} \label{refd1}
		Let $k$ be a number field such that $\rg(A(k))=2$. Assume that the Hilbert $2$-class field of $k$ stops at the first layer.
		If moreover  $k$ has an unramified quadratic extension with cyclic $2$-class group, then the $4$-$\rg$ of $A(k)$ equals $1$. 
	\end{lemma}
	\begin{proof}Let $G_{k}$ denotes the Galois group of the maximal unramified $2$-extension  $\mathcal{L}(k)$ of $k$ over $k$.
		As the Hilbert $2$-class field tower of $k$ stops at the first layer,  
		$G_{k}$ is abelian.
		Assume that $4$-$\mathrm{rank}(A(k))=2$. 
		Thus, there exist  two parameters $a$ and $b$  (cf. \cite[Lemma 1]{Ben17})  such that  $G_{k} =\langle a,b \rangle$ and such that the three subgroups of $G_{k}$ of index $2$ are 
		\begin{center}
			$H_{1}=\langle a, b^2, G_{k}'\rangle=\langle a, b^2 \rangle$,
			
			$H_{2}=\langle ab, b^2, G_{k}'\rangle=\langle ab, b^2 \rangle$ \quad  and \quad	$H_{3}=\langle a^2, b, G_{k}'\rangle=\langle a^2, b  \rangle.$
		\end{center}
		Thus all the three unramified quadratic extensions of $k$ have $2$-class groups of rank $2$, which is a contradiction. 
	\end{proof}

	\begin{lemma} \label{refd2}Let $n\geq 1$.
		The field  $F_{\nu,n}  $  has an unramified quadratic extension with cyclic $2$-class group if and only if  $q\equiv    7\pmod 8$.
	\end{lemma}
	\begin{proof}
		Notice that the unramified quadratic extensions of $F_{\nu,n}$ are $K_{\nu,n}$, $K_{\nu,n}' $ and $K_{\nu,n}''$ the  $n$th layer of the cyclotomic extension of 
		$K_\nu=\mathbb{Q}(\sqrt{\nu q},\sqrt{rs})$, $K_\nu'=\mathbb{Q}(\sqrt{\nu r},\sqrt{qs})$ and $K_\nu''=\mathbb{Q}(\sqrt{\nu s},\sqrt{qr})$ respectively. By \cite[Theorem 1.4]{ChemseddinGreenbergConjectureI} if $q\equiv 3\pmod 8$  we have
		$\rg(A(K_{\nu,n}'))=\rg(A(K_{\nu,n}''))=2$. Moreover, if $q\equiv 7\pmod 8$ then
		the $2$-class groups  of the extensions  $K_{\nu,n}'$ 
		and $K_{\nu,n}''$  are cyclic. So the lemma.
	\end{proof}
	\begin{lemma}\label{lemgen}
		Let $q\equiv    7\pmod 8$  and $r\equiv s\equiv     3\pmod 8$  be   three distinct prime numbers such that  $\left(\frac{q}{s}\right)=\left(\frac{q}{r}\right)=(-1)^{\delta_{\nu, 2}}$ 
		and $\left(\frac{s}{r}\right)=1$. Let $\rho \in\{1,2\}$ with $\rho\not= \nu$.
		\begin{enumerate}[\rm1)]
			\item Let $a$   and $b$  be the integers such that
			$ \varepsilon_{\nu qrs}=a+b\sqrt{\nu  qrs}$. Then
			out of the three integers   
			$2^{\delta_{\nu, 1}}r(a+(-1)^{\delta_{\nu, 1}})$, $2^{\delta_{\nu, 2}}q(a-1)$ or $2^{\delta_{\nu, 2}}s(a+(-1)^{\delta_{\nu, 1}})$, exactly one of them is   a square in $\NN $. Furthermore, we have:
			\begin{enumerate}[$ a)$]
				\item  If $2^{\delta_{\nu, 1}}r(a+(-1)^{\delta_{\nu, 1}})$ is a square in $\NN $, then
				$$\sqrt{\nu \varepsilon_{\nu qrs}}= b_1\sqrt{r} +b_2\sqrt{\nu qs}  \quad \text{ and } \quad  2^{\delta_{\nu, 2}}=(-1)^{\delta_{\nu, 1}}rb_1^2+(-1)^{\delta_{\nu, 2}}\nu qsb_2^2.$$

				\item  If $2^{\delta_{\nu, 2}}q(a-1)$ is a square in $\NN $, then  
				$$\sqrt{2\varepsilon_{\nu qrs}}= b_1\sqrt{\nu q } +b_2\sqrt{rs}  \quad \text{ and } \quad  2=-2^{\delta_{\nu, 2}}qb_1^2+ rsb_2^2.$$
				
				\item If $2^{\delta_{\nu, 2}}s(a+(-1)^{\delta_{\nu, 1}})$ is a square in $\NN $, then 
				$$\sqrt{2\varepsilon_{\nu qrs}}= b_1\sqrt{\nu s } +b_2\sqrt{qr}  \quad \text{ and } \quad  2=(-1)^{\delta_{\nu, 1}}2^{\delta_{\nu, 2}}sb_1^2+
				(-1)^{\delta_{\nu, 2}} qrb_2^2.$$

			\end{enumerate}
			Here $b_1$ and $b_2$ are two integers.	
			\item  Let $x$ and $y$   be the integers such that
			$\varepsilon_{\rho qrs}=x+y\sqrt{\rho qrs}$. Then $2^{\delta_{\nu, 2}}q(x-1)$  is a square in $\NN $. Moreover, we have:
			$$\sqrt{\rho\varepsilon_{\rho qrs}}=  y_1\sqrt{q} +y_2\sqrt{\rho  rs}  \text{ and } \rho=-qy_1^2+\rho rsy_2^2.$$
		\end{enumerate}
	\end{lemma}	
	
	\begin{proof}
		Let us first assume that $\nu=2$ $($i.e. $\rho=1)$.
		\begin{enumerate}[$\bullet$]

			\item As $N(\varepsilon_{2qrs})=1$,  by the unique factorization  of $a^{2}-1=2qrsb^{2}$ in $\mathbb{Z}$, and \cite[Lemma 5]{Az-00}, there exist $b_1$ and $b_2$ in $\mathbb{Z}$ such that  we have one of the following systems:
			
			$$(1):\ \left\{ \begin{array}{ll}
				a\pm1=2qb_1^2\\
				a\mp1=rsb_2^2,
			\end{array}\right.  \quad
			(2):\ \left\{ \begin{array}{ll}
				a\pm1=2rb_1^2\\
				a\mp1=qsb_2^2,
			\end{array}\right. \quad
			(3):\ \left\{ \begin{array}{ll}
				a\pm1=2sb_1^2\\
				a\mp1=qrb_2^2,
			\end{array}\right. 
			$$

			$$   
			(4):\ \left\{ \begin{array}{ll}
				a\pm1=qb_1^2\\
				a\mp1=2rsb_2^2,
			\end{array}\right. \quad (5):\ \left\{ \begin{array}{ll}
				a\pm1=rb_1^2\\
				a\mp1=2qsb_2^2
			\end{array}\right. \quad (6):\ \left\{ \begin{array}{ll}
				a\pm1=sb_1^2\\
				a\mp1=2qrb_2^2.
			\end{array}\right. $$ $$ \text{ or } (7): \ \left\{ \begin{array}{ll}
				a\pm1=b_1^2\\
				a\mp1=2qrsb_2^2.
			\end{array}\right.
			$$
			\begin{enumerate}[$\star$]
				\item  Assume that the system $\left\{ \begin{array}{ll}
					a+1=2qb_1^2\\
					a-1=rsb_2^2,
				\end{array}\right.$ holds. We have:	
				\[1=\left(\dfrac{2qb_1^2}{r}\right)=\left(\dfrac{a+1}{r}\right)=\left(\dfrac{a-1+2}{r}\right)=\left(\dfrac{rsb_2^2+2}{r}\right)=\left(\dfrac{2}{r}\right)=-1,
				\]
				which is absurd. So this system can not hold.
				
				\item  Assume that the system $\left\{ \begin{array}{ll}
					a\pm1=2rb_1^2\\
					a\mp1=qsb_2^2,
				\end{array}\right.$ holds. We have:	
				\[1=\left(\dfrac{2rb_1^2}{s}\right)=\left(\dfrac{a+1}{s}\right)=\left(\dfrac{a-1+2}{s}\right)=\left(\dfrac{qsb_2^2+2}{s}\right)=\left(\dfrac{2}{s}\right)=-1,
				\] and 
				\[1=\left(\dfrac{2rb_1^2}{q}\right)=\left(\dfrac{a-1}{q}\right)=\left(\dfrac{a+1-2}{q}\right)=\left(\dfrac{qsb_2^2-2}{q}\right)=\left(\dfrac{-2}{q}\right)=-1,
				\]
				
				which is absurd. So this system can not hold.
				\item  Assume that the system $\left\{ \begin{array}{ll}
					a-1=2sb_1^2\\
					a+1=qrb_2^2,
				\end{array}\right.$ holds. We have:	
				\[1=\left(\dfrac{qrb_2^2}{s}\right)=\left(\dfrac{a+1}{s}\right)=\left(\dfrac{a-1+2}{s}\right)=\left(\dfrac{rsb_2^2+2}{s}\right)=\left(\dfrac{2}{s}\right)=-1,
				\]
				which is absurd. So this system can not hold.
				\item  Assume that the system $\left\{ \begin{array}{ll}
					a\pm1=qb_1^2\\
					a\mp1=2rsb_2^2,
				\end{array}\right.$ holds. We have:	
				\[1=\left(\dfrac{2rsb_2^2}{q}\right)=\left(\dfrac{a-1}{q}\right)=\left(\dfrac{a+1-2}{q}\right)=\left(\dfrac{qb_1^2-2}{q}\right)=\left(\dfrac{-2}{q}\right)=-1,
				\] and 
				\[-1=\left(\dfrac{qb_1^2}{r}\right)=\left(\dfrac{a-1}{r}\right)=\left(\dfrac{a+1-2}{r}\right)=\left(\dfrac{2rsb_2^2-2}{r}\right)=\left(\dfrac{-2}{r}\right)=1,
				\]
				which is absurd. So this system can not hold.
				\item  Assume that the system $\left\{ \begin{array}{ll}
					a-1=rb_1^2\\
					a+1=2qsb_2^2,
				\end{array}\right.$ holds. We have:	
				\[1=\left(\dfrac{rb_1^2}{q}\right)=\left(\dfrac{a-1}{q}\right)=\left(\dfrac{a+1-2}{q}\right)=\left(\dfrac{2qsb_2^2-2}{q}\right)=\left(\dfrac{-2}{q}\right)=-1,
				\]
				which is absurd. So this system can not hold.
				\item  Assume that the system $\left\{ \begin{array}{ll}
					a\pm1=sb_1^2\\
					a\mp1=2qrb_2^2,
				\end{array}\right.$ holds. We have:	
				\[1=\left(\dfrac{sb_1^2}{r}\right)=\left(\dfrac{a+1}{r}\right)=\left(\dfrac{a-1+2}{r}\right)=\left(\dfrac{2qrb_2^2+2}{r}\right)=\left(\dfrac{2}{r}\right)=-1,
				\] and 
				\[1=\left(\dfrac{sb_1^2}{q}\right)=\left(\dfrac{a-1}{q}\right)=\left(\dfrac{a+1-2}{q}\right)=\left(\dfrac{2qrb_2^2-2}{q}\right)=\left(\dfrac{-2}{q}\right)=-1,
				\]
				which is absurd. So this system can not hold.
				\item  Assume that the system $\left\{ \begin{array}{ll}
					a\pm1=b_1^2\\
					a\mp1=2qrsb_2^2,
				\end{array}\right.$ holds. We have:	
				\[1=\left(\dfrac{b_1^2}{r}\right)=\left(\dfrac{a+1}{r}\right)=\left(\dfrac{a-1+2}{r}\right)=\left(\dfrac{2qrsb_2^2+2}{r}\right)=\left(\dfrac{2}{r}\right)=-1,
				\] and 
				\[1=\left(\dfrac{b_1^2}{q}\right)=\left(\dfrac{a-1}{q}\right)=\left(\dfrac{a+1-2}{q}\right)=\left(\dfrac{2qrsb_2^2-2}{q}\right)=\left(\dfrac{-2}{q}\right)=-1,
				\]
				which is absurd. So this system can not hold.
			\end{enumerate}
			Then The only consistent remaining cases are : $$(1):\ \left\{ \begin{array}{ll}
				a-1=2qb_1^2\\
				a+1=rsb_2^2,
			\end{array}\right.  \quad
			(3):\ \left\{ \begin{array}{ll}
				a+1=2sb_1^2\\
				a-1=qrb_2^2,
			\end{array}\right. \quad
			(5):\ \left\{ \begin{array}{ll}
				a+1=rb_1^2\\
				a-1=2qsb_2^2,
			\end{array}\right. 
			$$
			Using these systems, we construct $\sqrt{\varepsilon_{2qrs}}$ in every case.

			\item For the second item, we proceed similarly and eliminate all systems except $\left\{ \begin{array}{ll}
				x-1=2qy_1^2\\
				x+1=2rsy_2^2,
			\end{array}\right. $ with $y_1$ and $y_2$ are two integers such that $y= 2y_1y_2$. So  gives the second item for $\nu=2$ $($i.e. $\rho=1)$.
		\end{enumerate}
		
		The proof is analogous for the case $\nu=1$.
	\end{proof}
	\begin{corollary}\label{corr}	Let $q\equiv    7\pmod 8$  and $r\equiv s\equiv     3\pmod 8$  be   three distinct prime numbers such that  $\left(\frac{q}{s}\right)=\left(\frac{q}{r}\right)=(-1)^{\delta_{\nu, 2}}$ 
		and $\left(\frac{s}{r}\right)=1$. Then  
		$h_2(F_{\nu,1})= h_2(\nu qrs)$ or $2h_2(\nu qrs)$. More precisely, we have:  
		$$h_2(F_{\nu,1})= h_2(\nu qrs) \text{ if and only if }  2^{\delta_{\nu, 2}}q(a-1)\not\in \NN^2.$$
		
	\end{corollary}	
	\begin{proof}
		First notice that according to the previous lemma and Wada's method, a fundamental system of units of $F_{\nu,1}=\mathbb{Q}(\sqrt{qrs},\sqrt{2})$  is
		$\{\varepsilon_2,\varepsilon_{qrs},\sqrt{\varepsilon_{qrs}\varepsilon_{2qrs}} \}$ or $\{\varepsilon_2,\varepsilon_{qrs}, \varepsilon_{2qrs} \}$ according to whether $2^{\delta_{\nu, 2}}q(a-1)\in \NN^2$ or not. Thus, 
		$q(F_{\nu,1})=1$ or $2$ and 
		$$q(F_{\nu,1}) =1\text{ if and only if }  2^{\delta_{\nu, 2}}q(a-1)\not\in \NN^2.$$

		Notice that the class number formula gives	$h_2(F_{\nu,1})=\frac{q(F_{\nu,1})}{4} h_2(qrs)h_2({2qrs}).$
		On the other hand, we have    
		$h_2(qrs)h_2({2qrs})=4h_2({ \nu qrs})$ 
		(in fact one can proceed as in \cite[p. 164]{BenjShn(22)} by using Redei matrix method to get this or   one can consult   \cite[p. 40]{AzmouhcapitulationActaArith})). Thus,
		$h_2(F_{\nu,1})=q(F_{\nu,1})h_2(\nu qrs)$, hence the result.
	\end{proof}

	\medskip
	
	Now we can state the main result of this subsection.
	
	\medskip
	
	\begin{theorem}[{\bf The Third Main Theorem}]\label{etacorolaK} 
		Let $q\equiv    7\pmod 8$  and $r\equiv s\equiv     3\pmod 8$  be   three distinct prime numbers such that  $\left(\frac{q}{s}\right)=\left(\frac{q}{r}\right)=(-1)^{\delta_{\nu, 2}}$ 
		and $\left(\frac{s}{r}\right)=1$. 	Put      $ \varepsilon_{\nu qrs}=a+b\sqrt{\nu qrs}$, where $a$ and $b$ are integers.
		Assume   that  $2^{\delta_{\nu,2}}q(a-1)$   is {\bf not} a square in $\NN$. Then the following holds:
		\begin{enumerate}[$1)$]
			\item   Letting $m$ be the positive integer such that $h_2(\nu qrs)=2^m$, we have	 
			$$A_\infty(K_\nu)\simeq A(K_\nu)\simeq\ZZ/2 \ZZ\times\ZZ/2^{m-2} \ZZ .$$
			\item 	
			The Galois group of $\mathcal{L}(K_{\nu,\infty})/K_{\nu,\infty}$, with  $\mathcal{L}(K_{\nu,\infty})$ is the maximal unramified pro-$2$-extension    of $K_{\nu,\infty}$, is abelian.  
		\end{enumerate}
	\end{theorem}

	\medskip
	
Before giving the proof of this theorem, we first prove the following lemma.

	\medskip

	\begin{lemma}\label{lem2}
		Let $q\equiv    7\pmod 8$  and $r\equiv s\equiv     3\pmod 8$  be   three distinct prime numbers such that  $\left(\frac{q}{s}\right)=\left(\frac{q}{r}\right)=(-1)^{\delta_{\nu, 2}}$ 
	and $\left(\frac{s}{r}\right)=1$. 	Let       $ \varepsilon_{\nu qrs}=a+b\sqrt{\nu qrs}$, where $a$ and $b$ are integers, and let   $\rho \in\{1,2\}$ with $\rho\not= \nu$. Assume that $2^{\delta_{\nu, 2}}q(a-1)$ is {\bf not} a square in $\NN$.
		$$E_{K_{\nu,1}} =\langle-1,   \varepsilon_2,  \varepsilon_{rs}     ,\sqrt{\varepsilon_{q}}, \sqrt{\varepsilon_{2q}}, \sqrt{\varepsilon_{2rs}},\sqrt{\varepsilon_{rs}\varepsilon_{\nu qrs}}, \sqrt{ \varepsilon_{\rho qrs}} \rangle.$$
		Thus, $q(K_{\nu,1})=2^5$.
	\end{lemma}	
	\begin{proof}
		We prove our lemma in the case where $\nu=2$, for the other case we proceed similarly. Let $ q\equiv     7\pmod 8$ and  $r\equiv s\equiv    5\pmod 8$   be   prime numbers such that $\left(\frac{q}{r}\right)=\left(\frac{q}{s}\right)=  -1 $ and 
		$\left(\frac{s}{r}\right)=1$. We shall use the method of Wada  presented in Page  \pageref{algo wada}. Consider the three  biquadratic subfields of $K_{\nu,1}$ defined by
		$k_1=\mathbb{Q}(\sqrt{2},  \sqrt{q})$, $k_2=\mathbb{Q}(\sqrt{2},  \sqrt{rs})$ and $k_3=\mathbb{Q}(\sqrt{2},  \sqrt{qrs})$. 
		Let $\tau_1$, $\tau_2$ and $\tau_3$ be the elements of  $ \mathrm{Gal}(K_{\nu,1}/\QQ)$ defined by
		\begin{center}	\begin{tabular}{l l l }
				$\tau_1(\sqrt{2})=-\sqrt{2}$, \qquad & $\tau_1(\sqrt{q})=\sqrt{q}$, \qquad & $\tau_1(\sqrt{rs})=\sqrt{rs},$\\
				$\tau_2(\sqrt{2})=\sqrt{2}$, \qquad & $\tau_2(\sqrt{q})=-\sqrt{q}$, \qquad &  $\tau_2(\sqrt{rs})=\sqrt{rs},$\\
				$\tau_3(\sqrt{2})=\sqrt{2}$, \qquad &$\tau_3(\sqrt{q})=\sqrt{q}$, \qquad & $\tau_3(\sqrt{rs})=-\sqrt{rs}.$
			\end{tabular}
		\end{center} 
		Notice that  $\mathrm{Gal}(K_1/\QQ)=\langle \tau_1, \tau_2, \tau_3\rangle$
		and that the subfields  $k_1$, $k_2$ and $k_3$ are
		fixed by  $\langle \tau_3\rangle$, $\langle\tau_2\rangle$ and $\langle\tau_2\tau_3\rangle$ respectively. Therefore according to Wada's algorithm a fundamental system of units  of $K_{\nu,1}$ consists  of seven  units chosen from those of $k_1$, $k_2$ and $k_3$, and  from the square roots of the elements of $E_{k_1}E_{k_2}E_{k_3}$ which are squares in $K_1$.	  
		We have:  
		\begin{enumerate}[$\bullet$]
			\item A fundamental system of units  of 
			$k_1$ is $\{\varepsilon_2,\sqrt{\varepsilon_{q}}, \sqrt{\varepsilon_{2q}}\}$. In fact, we have $ \sqrt{2\varepsilon_{q}}=\alpha_1 +\alpha_2\sqrt{q})$ and $ \sqrt{\varepsilon_{2q}}=\alpha_3\sqrt{2}+\alpha_4\sqrt{q}$ for some integers $\alpha_i$  (cf. \cite[Lemma 6]{BenSnyder25PartII} or \cite{Az-00}).
			
			\item A fundamental system of units  of 
			$k_2$ is $\{\varepsilon_2,\sqrt{\varepsilon_{rs}}, \sqrt{\varepsilon_{2rs}}\}$. In fact, we have $ \sqrt{2\varepsilon_{2rs}}=\beta_1+\beta_2\sqrt{2rs}$ and $ \sqrt{\varepsilon_{rs}}=\beta_3\sqrt{r}+\beta_4\sqrt{s}$ for some integers  $\beta_i$  (cf. \cite[Lemma 4]{aczIJM}).

			\item By Corollary \ref{corr} a fundamental system of units of $k_3$ is   $\{\varepsilon_2,\varepsilon_{qrs}, \varepsilon_{2qrs}\}$ since $2q(a-1)$ is not a square in $\NN$, where  $a$   and $b$  are the integers such that $ \varepsilon_{ 2qrs}=a+b\sqrt{2 qrs}$.
		\end{enumerate}

		So 
		\begin{equation}
			E_{k_1}E_{k_2}E_{k_3}=\langle-1,  \varepsilon_{2}, \varepsilon_{rs},\varepsilon_{qrs}, \varepsilon_{2qrs} ,\sqrt{\varepsilon_{q}},\sqrt{\varepsilon_{2q}},\sqrt{\varepsilon_{2rs}}\rangle.
		\end{equation}
		
		Notice that according to Lemma \ref{lemgen} we have $\sqrt{\varepsilon_{ qrs}}\in K_{\nu,1}$. 
		Let us find the elements $\chi$   of $ K_{\nu,1}$ which are the  square root of an element of $E_{k_1}E_{k_2}E_{k_3}$. 
		Therefore, we can assume that		
		\begin{equation}\label{equa1}
			\chi^2=\varepsilon_{2}^a\varepsilon_{rs}^b \varepsilon_{2qrs}^c\sqrt{\varepsilon_{q}}^d\sqrt{\varepsilon_{2q}}^e\sqrt{\varepsilon_{2rs}}^f
		\end{equation}
		
		where $a, b, c, d, e$ and $f$ are in $\{0, 1\}$. We shall use the norm maps from $K$ to its subfields to eliminate  the equations  which do not occur.  
		
		\noindent\ding{229} By   applying the norm $N_{K_{\nu,1}/\QQ(\sqrt{q},\sqrt{rs})}=1+\tau_1$, we get:
		\begin{eqnarray*}
			N_{K_{\nu,1}/\QQ(\sqrt{q},\sqrt{rs})}(\chi^2)&=& (-1)^a \cdot\varepsilon_{rs}^{2b}\cdot (-\varepsilon_{q})^d\cdot (-1)^{e}. 	\end{eqnarray*}
		Thus $a+d+ e \equiv 0\pmod2$. As $\varepsilon_q$ is not a square in $\QQ(\sqrt{q},\sqrt{rs})$, we get $d=0$. Then  $a=e$.
		Therefore, we have:
		\begin{equation*} 
			\chi^2=\varepsilon_{2}^a\varepsilon_{rs}^b \varepsilon_{2qrs}^c\sqrt{\varepsilon_{2q}}^a\sqrt{\varepsilon_{2rs}}^f.
		\end{equation*}
		\noindent\ding{229} By   applying the norm $N_{K_{\nu,1}/\QQ(\sqrt{2q},\sqrt{rs})}=1+\tau_1\tau_2$, we get:
		\begin{eqnarray*}
			N_{K_{\nu,1}/\QQ(\sqrt{2q},\sqrt{rs})}(\chi^2)&=& (-1)^a \cdot\varepsilon_{rs}^{2b}\cdot (-\varepsilon_{2q})^f.
		\end{eqnarray*}
		As $\varepsilon_{2q}$ is not a square in $\QQ(\sqrt{2q},\sqrt{rs})$, we get $a=0$. Thus, we have:
		\begin{equation*}
			\chi^2=\varepsilon_{rs}^b \varepsilon_{2qrs}^c\sqrt{\varepsilon_{2rs}}^f.
		\end{equation*}
		
		\noindent\ding{229} By   applying the norm $N_{K_{\nu,1}/\QQ(\sqrt{q},\sqrt{2rs})}=1+\tau_1\tau_3$, we get:
		\begin{eqnarray*}
			N_{K_{\nu,1}/\QQ(\sqrt{q},\sqrt{2rs})}(\chi^2)&=& \varepsilon_{2qrs}^{2c}\cdot (-\varepsilon_{2rs})^f.
		\end{eqnarray*}
		Thus we have $f=0$ and :
		\begin{equation*}
			\chi^2=\varepsilon_{rs}^b \varepsilon_{2qrs}^c.
		\end{equation*}
		Notice that  $\sqrt{\varepsilon_{rs}}, \sqrt{\varepsilon_{2 qrs}} \notin K_{\nu,1}$ (cf. Lemma \ref{lemgen}) and  $\sqrt{\varepsilon_{rs}}\sqrt{\varepsilon_{2 qrs}}\in K_{\nu,1}$  (cf. Lemma \ref{lemgen}). Thus, the only elements of $E_{k_1}E_{k_2}E_{k_3}$ which are squares in $K_{\nu,1}$ are ${\varepsilon_{rs}}{\varepsilon_{\eta qrs}}$ and ${\varepsilon_{qrs}}$. Hence,

		$$E_{K_{\nu,1}}=\langle -1, \varepsilon_2,\varepsilon_{rs}, \varepsilon_{2qrs}, \sqrt{\varepsilon_q},\sqrt{\varepsilon_{2q}},\sqrt{\varepsilon_{2rs}}, \sqrt{\varepsilon_{ qrs}},\sqrt{\varepsilon_{rs}\varepsilon_{ 2qrs}} \rangle.$$

	\end{proof}

	\begin{proof}[{\bf Proof of Theorem \ref{etacorolaK}}]  As  $h_2(\nu q)=h_2(rs)=1$ (cf. \cite[Corollary 18.4]{connor88}), we have
		$$h_2(K_\nu)=\frac{1}{4}q(K_\nu)  h_2(\nu qrs).$$
		
		By making use of Lemma \ref{lemgen} and Wada's method, we check that $\{\varepsilon_{\nu q},\varepsilon_{rs},\varepsilon\} $, with 
		$\varepsilon=\sqrt{\varepsilon_{rs}\varepsilon_{\nu qrs}}$ or $\sqrt{\varepsilon_{\nu q}\varepsilon_{\nu qrs}}$, is a fundamental system of units of $K_\nu$.
		Thus $q(K_\nu)=2$ and so $h_2(K_\nu)=\frac{1}{2} h_2(\nu qrs)$.
		Moreover, by applying the  class number formula  to $K_{\nu,1}$,    we get:
		\begin{eqnarray*} 
			h_2(K_{\nu,1})&=&\frac{1}{2^{9}}q(K_{\nu,1})  h_2(q) h_2(2q )h_2(rs) h_2(2rs)h(2)  h_2(qrs) h_2(2qrs)= \frac{1}{2} h_2(\nu qrs) \nonumber 
		\end{eqnarray*}
		In fact, we have  $q(K_{\nu,1})=2^5$ (cf.  Lemma \ref{lem2}), $h_2(q) =h_2(2q )=h(2)=h_2(rs) =1$ (cf. \cite[Corollary 18.4]{connor88}) and $h_2(2rs)=2 $ (cf.  \cite[Corollary 19.7]{connor88}) and $h_2(qrs)h_2({2qrs})=4h_2({ \nu qrs})$ (cf. the proof of Corollary \ref{corr}).
		
		Let $n\geq0$ be an integer.	It follows by Fukuda's theorem, we have $h_2(K_{\nu, n})=\frac{1}{2}h_2(\nu qrs) $. 
		Therefore,   according to  Corollary \ref{corr}, $h_2(K_{\nu, n})=  \frac{1}{2}h_2(F_{\nu, n})$ and so by Proposition \ref{LemBenjShn} the Hilbert $2$-class field tower of 
		$K$ stops at the first layer.

		Notice that by Lemma \ref{refd2}, $F_{\nu, n}$ admits an unramified quadratic extension of   cyclic $2$-class group, then Lemma \ref{refd1} gives $A_\infty(F_\nu)\simeq A(F_\nu)\simeq\ZZ/2 \ZZ\times\ZZ/2^{m-1} \ZZ  $, for all $n\geq0$.
		It follows that, by   for all $n\geq0$,  $A(K_n)\simeq\ZZ/2 \ZZ\times\ZZ/2^{m-2} \ZZ$   (cf. Theorem \ref{AabounePrzekhini}), which completes the proof.
		
	\end{proof}
	
	The following is a direct deduction from the above proof and Theorem \ref{AabounePrzekhini}.
	\begin{corollary}
			Let $q\equiv    7\pmod 8$  and $r\equiv s\equiv     3\pmod 8$  be   three distinct prime numbers such that  $\left(\frac{q}{s}\right)=\left(\frac{q}{r}\right)=(-1)^{\delta_{\nu, 2}}$ 
		and $\left(\frac{s}{r}\right)=1$. 	Put      $ \varepsilon_{\nu qrs}=a+b\sqrt{\nu qrs}$, where $a$ and $b$ are integers.
		Assume   that  $2^{\delta_{\nu,2}}q(a-1)$   is {\bf not} a square in $\NN$. Then 
		$$A_\infty(K_\nu')\simeq A(K_\nu')\simeq A_\infty(K_\nu'')\simeq A(K_\nu'') \simeq   \ZZ/2^{m-1} \ZZ.$$
	 Here $K_\nu'=\mathbb{Q}(\sqrt{\nu r},\sqrt{qs})$, $K_\nu''=\mathbb{Q}(\sqrt{\nu s},\sqrt{qr})$ and $m$ is the positive integer such that $h_2(\nu qrs)=2^m$.
	\end{corollary}

		


	\section*{\bf Acknowledgment}
 	We would  like to thank Katharina M\"uller for   her useful feedback this paper.

\end{document}